\documentclass{amsart}
\usepackage{latexsym}
\usepackage{amsthm,amsmath,amssymb}
\usepackage{fullpage}
\usepackage{float}
\usepackage{amsfonts}
\usepackage{graphicx}
\usepackage{amsmath}

\usepackage[utf8x]{inputenc}
\usepackage{amssymb}
\usepackage{epsfig}
\usepackage{color}

\newcommand{\R}{\mathbb{R}}

\newtheorem{theorem}{Theorem}[section]
\newtheorem{corollary}{Corollary}[section]

\newtheorem{lemma}[theorem]{Lemma}
\newtheorem{proposition}[theorem]{Proposition}

\theoremstyle{definition}

\newtheorem{remark}{Remark}[section]

\def\var{\varepsilon}

\newcommand{\gr}[1]{{
#1}}

\title{Asymptotic profile in selection-mutation equations: Gauss versus Cauchy distributions}

\author[A. Calsina]{\`Angel Calsina}

\address[A.C.]{Departament de Matemàtiques, Universitat Autònoma de Barcelona, 08193 Bellaterra (Barcelona), Spain}
\email{acalsina@mat.uab.cat}

\author[S. Cuadrado]{S\'ilvia Cuadrado}
\address[S.C.]{Departament de Matemàtiques, Universitat Autònoma de Barcelona, 08193 Bellaterra (Barcelona), Spain}
\email{silvia@mat.uab.cat}

\author[L. Desvillettes]{Laurent Desvillettes}
\address[L.D.]{Univ. Paris Diderot, Sorbonne Paris Cit\'e, Institut de Math\'ematiques de Jussieu - Paris Rive Gauche, UMR 7586, CNRS, Sorbonne Universit\'es, UPMC Univ. Paris 06, F-75013, Paris, France.}
\email{desvillettes@math.univ-paris-diderot.fr}

\author[G. Raoul]{Ga\"el Raoul}
\address[G.R.]{CMAP, Ecole Polytechnique, CNRS, Université Paris-Saclay, Route de Saclay, 91128 Palaiseau cedex, France.}
\email{raoul@cmap.polytechnique.fr}

\begin{document}

\maketitle

\begin{abstract}
In this paper, we study the asymptotic (large time) behavior of a selection-mutation-competition
model for a population structured with respect to a phenotypic trait, when the rate of mutation is very small. We assume that the reproduction is asexual, and that the mutations can be described by a linear integral operator. We are interested in the interplay between the time variable $t$ and the rate $\var$ of mutations. We show that depending on $\alpha > 0$, the limit $\var \to 0$ with $t = \var^{-\alpha}$ can lead
to population number densities
which are either Gaussian-like (when $\alpha$ is \gr{small}) or Cauchy-like
(when $\alpha$ is \gr{large}).
\end{abstract}

\bigskip

\section{Introduction}

\subsection{Selection-mutation-competition models}\label{subsec:sel-mut-comp}

\gr{The phenotypic diversity of a species impacts its ability to evolve. In particular, the importance of the variance of the population along a phenotypic trait is illustrated by the \emph{fundamental theorem of natural selection} \cite{Fisher}, and the \emph{breeder's equation} \cite{Lush}: the evolution speed of a population along a one dimensional fitness gradient (or under artificial selection) is proportional to the variance of the initial population. Recently, the phenotypic variance of populations has also come to light as an important element to describe the evolutionary dynamics of ecosystems (where many interacting species are considered) \cite{Violle,Bolnick,Vellend}.}

%

\medskip


\gr{Over the last decade, the thematic of \emph{Evolutionary Rescue} has emerged as an important question \cite{Bell,Carlson,Gonzales} (see also the seminal work of Luria and Delbr\"uck \cite{Luria}), and led to a new interest in the phenotypic distribution of populations, beyond phenotypic variance}. Evolutionary Rescue is concerned with a population living in an environment that changes suddenly. The population will survive either if some individuals in the population carry an unusual trait that turns out to be successful in the new environment, or if new mutants able to survive in the new environment appear before the population goes extinct (see \cite{Martin} for a discussion on the relative effect of \emph{de novo mutations} and \emph{standing variance} in Evolutionary Rescue). In any case, the fate of the population will not be decided by the properties of the bulk of its density, but rather by the properties of the tail of the initial distribution of the populations, close to the favourable traits for the new environment. A first example of such problem comes from emerging disease \cite{Gandon}: Animal infections sometimes are able to infect humans. This phenomena, called zoonose, is the source of many human epidemics: HIV, SARS, Ebola, MERS-CoV, etc. A zoonose may happen if a pathogen that reaches a human has the unusual property of being adapted \gr{to this new human host.} A second example comes from the emergence of microbes resistant to an antimicrobial drug that is suddenly spread in the environment of the microbe. This second phenomenon can easily be tested experimentally \cite{Bell,Toprak}, and \gr{has major public health implications} \gr{\cite{Canton}}.


\gr{Most papers devoted to the genetic diversity of populations structured by a continuous phenotypic trait describe the properties of mutation-selection equilibria. It is however also interesting to describe the genetic diversity of population that are not at equilibrium (\emph{transient dynamics}):} pathogen populations for instance are often in transient situations, either invading a new host, or being eliminated by the immune system. We refer to \cite{Hastings} for a review on transient dynamics in ecology. For asexual populations \gr{structured by a continuous phenotypic trait}, several models exist, corresponding to different biological assumptions \cite{Champagnat}. If the mutations are modeled by a diffusion, the steady populations \gr{(for a model close to \eqref{eqq0}, but where mutations are modelled by a Laplacian)} are Gaussian distributions \cite{Kimura, Burger}. \gr{Furthermore, \cite{Alfaro,Coville} have considered some transient dynamics for this model. In the model that we will consider (see \eqref{eqq0}), the mutations are modelled by a non-local term. It was shown in \cite{Burger2} (see also \cite{Burger}) that mutation-selection equilibria are then Cauchy profiles (under some assumptions), and this result has been extended to more general mutation kernels in \cite{Calsina}, provided that the mutation rate is small enough.} Finally, let us notice that the case of sexual population is rather different, since recombinations by themselves can imply that a \emph{mutation-recombination equilibrium} exists, even without selection. We refer to the infinitesimal model \cite{Bulmer}, and to \cite{Turelli} for some studies on the phenotypic distribution of sexual species in a context close to the one presented here for asexual populations.

 \medskip

In this article, we consider a population consisting of individuals structured \gr{by} a quantitative phenotypic trait $x \in I$ ($I$ open interval of $\R$ containing $0$), and denote by $f : = f(t,x) \ge 0$ its density. Here, the trait $x$ is fully inherited by the offspring (if no mutation occurs), so that $x$ is indeed rather a breeding value than a phenotypic trait (see \cite{Mather}). We assume that the individuals reproduce with a rate $1$, and die at a rate
\[x^{2} + \int_{I}f(t,y)\,\mbox{d}y.\]
This means that the individuals with trait $x=0$ are those who are best adapted to their environment,
and that the fitness decreases like a parabola around this optimal trait (this is expected in the surroundings
of a trait of maximal fitness). It also means that the strength of the competition modeled by the logistic term
is identical for all traits. When an individual of trait $x\in I$ gives birth, we assume that the offspring will have the trait $x$ with probability $1-\varepsilon$, and a different trait $x'$ with probability $\var\in(0,1)$. $\var$ is then the probability that a mutation affects the phenotypic trait of the offspring. We can now define the growth rate of the population of trait $x$ (that is the difference between the rate of \emph{births without mutation}, minus the death rate) as
$$ r_\var(t,x) = 1-\varepsilon -x^{2} - \int_{I}f(t,y)\,\mbox{d}y. $$
When a mutation affects the trait of the offspring, we assume that the trait $x'$ of the mutated offspring is drawn from a law over the set of phenotypes $I\subset \mathbb R$ with a density $\gamma := \gamma (x)\in L^1(I)$. The function $\gamma$ then satisfies
\[\gamma(x)\ge 0,\quad \int_I \gamma(x)\, dx = 1,\]
and we assume moreover that $\gamma$ is bounded, $C^{1}$, with bounded derivative and strictly positive on $I$. The main assumption here is that the law of the trait of a mutated offspring does not depend of the trait of its parent. This classical assumption, known as \emph{house of cards} is not the most realistic, but it can be justified when the mutation rate is small \cite{Burger} (see also \cite{Calsina}\gr{)}. All in all, we end up with the following equation:
\begin{equation}
\label{eqq0}
 \frac{\partial f_\var(t,x)}{\partial t}= r_\var(t,x) \, f_\var(t,x) +\varepsilon \,\gamma(x)\,\int_{I}f_\var(t,y)\,\mbox{d}y.
\end{equation}
This paper is devoted to the study of the asymptotic behaviour of
the solutions of equation \eqref{eqq0} when $\varepsilon$ is small
and $t$ large and it is organized as follows. In the rest of Section
1 the main results are quoted, first in an informal way, and then as
rigourous statements. Section 2 contains the proof of Theorem
\ref{theorem1} and its corollary and finally, in Section 3, Theorem
\ref{thm:smallt} is proved.

\subsection{Asymptotic study of the model}\label{subsec:asymptotics}

When we consider the solutions of \eqref{eqq0}, two particular profiles naturally appear:
\begin{itemize}
\item \emph{A Cauchy profile:} For a given mutation rate $\var>0$ small enough, one expects that $f_\var(t,x)$ will converge, as $t$ goes to infinity, to the unique steady-state of \eqref{eqq0}, wich is the following Cauchy profile
\begin{equation}
\label{eqq2}
f_\var(\infty,x) :=  \frac{\var\, \gamma(x)\, \mathcal I_\var(\infty)}{\mathcal I_\var(\infty) - (1-\var) + x^2} ,
\end{equation}
where $\mathcal I_\var(\infty)$ is such that $\int_I f_\var(\infty,x)\,dx=\mathcal I_\var(\infty)$. 
This steady-state of \eqref{eqq0} is the so-called \emph{mutation-selection equilibrium} of the \emph{House of cards} model \eqref{eqq0}, which has been introduced in \cite{Burger2} (we also refer to \cite{Burger} for a broader presentation of existing results).
\smallskip

\item \emph{A Gaussian profile:} If $\var=0$, the solution of (\ref{eqq0}) can be written
\begin{equation}
\label{eqq1}
 f_0(t,x) = f(0,x) \, e^{- \int_0^t \mathcal I_0(s)\, ds + t - t\,x^2} ,
 \end{equation}
where $\mathcal I_0(t):=\int_I f_0(t,x)\,dx$, so that a Gaussian-like behavior (with respect to $x$) naturally appears in this case. Surprisingly, we are not aware of any reference to this property in the population genetics literature.
\end{itemize}

\medskip

We will show that, as suggested by the above arguments, we can describe the phenotypic distribution of the population, that is $x\mapsto f_\var(t,x)$, when either $t\gg 1$ (large time for a given mutation rate $\varepsilon>0$), or $0\leq\var\ll 1$ (small mutation rate, for a given time interval $t\in [0,T]$). Before providing the precise statements of our results (see Subsection~\ref{subsec:rigorous}), we will briefly describe them here, and illustrate them with numerical simulations. The numerical simulations presented in Fig.~\ref{fig1} and Fig.~\ref{fig2} are obtained thanks to a finite difference scheme (explicit Runge-Kutta in time), and we illustrate our result with a single simulation of \eqref{eqq0} with $\var=10^{-2}$, $I=[-3/2,3/2]$, $\gamma(x)=\frac 1{40\pi}e^{\frac{-x^2}{20}}$ and $f_\varepsilon(0,x)=\Gamma_2(\var,x-1)$ (see the definition of $\Gamma_2$ in eq. (\ref{eqq4}) below). The initial condition corresponds to a population at the mutation-selection equilibrium which environment suddenly changes (the optimal trait originally in $x=1$ moves to $x=0$ at $t=0$). This example is guided by the Evolutionary Rescue experiments described in Subsection~\ref{subsec:sel-mut-comp}, where the sudden change is obtained by the addition of e.g. salt or antibiotic to a bacterial culture.

%

\medskip

We describe two phases of the dynamics of the population:

\begin{itemize}
\item \emph{Large time: Cauchy profile.} We show that $f_\var(t,x)$ is asymptotically (when the mutation rate $\var>0$ is small) close to
\begin{equation}
\label{eqq4}
 \Gamma_2(\var,x) =\frac{\var\, \gamma(0)}{\gamma(0)^2\pi^2\,\var^2 + x^2},
\end{equation}
provided $t\gg \var^{-4}$\gr{. The population is then a time-independent Cauchy distribution for large times}. This theoretical result is coherent with \gr{numerical} results: we see in Fig.~\ref{fig1} that $f_\var(t,\cdot)$ is well described by $\Gamma_2(\var,\cdot)$, as soon as $t\geq 10^5$, which is confirmed by the value of $\|f_\var(t,\cdot)-\Gamma_2(\var,\cdot)\|_{L^1(I)}$ for $t\geq 10^5$ given by Fig.~\ref{fig2}.
\medskip

\item \emph{Short time: Gaussian profile.} We also show that $f_\var(t,x)$ is asymptotically (when the mutation rate $\var>0$ is small) close to
\begin{equation}
\label{eqq3}
 \Gamma_1(t,\var,x) = \frac{f(0,x)\, \sqrt t }{f(0,0)\int_I e^{-x^2}\,dx}e^{-x^2 \,t},
\end{equation}
provided $1\ll t\ll \var^{-2/3}$\gr{. The population has then a Gaussian-type distribution for} short (but not too short) times. This theoretical result is coherent with \gr{numerical} results: we see in Fig.~\ref{fig1} that $f_\var(t,\cdot)$ is well described by $\Gamma_1(t,\var,\cdot)$ for $t\in[10^2,10^4]$, which is confirmed by the value of $\|f_\var(t,\cdot)-\Gamma_2(\var,\cdot)\|_{L^1(I)}$ for $t\in[10^2,10^4]$ given by Fig.~\ref{fig2}.
\end{itemize}
%

%
%

\begin{figure}[tbp]
\begin{center}
\includegraphics[scale=0.3]{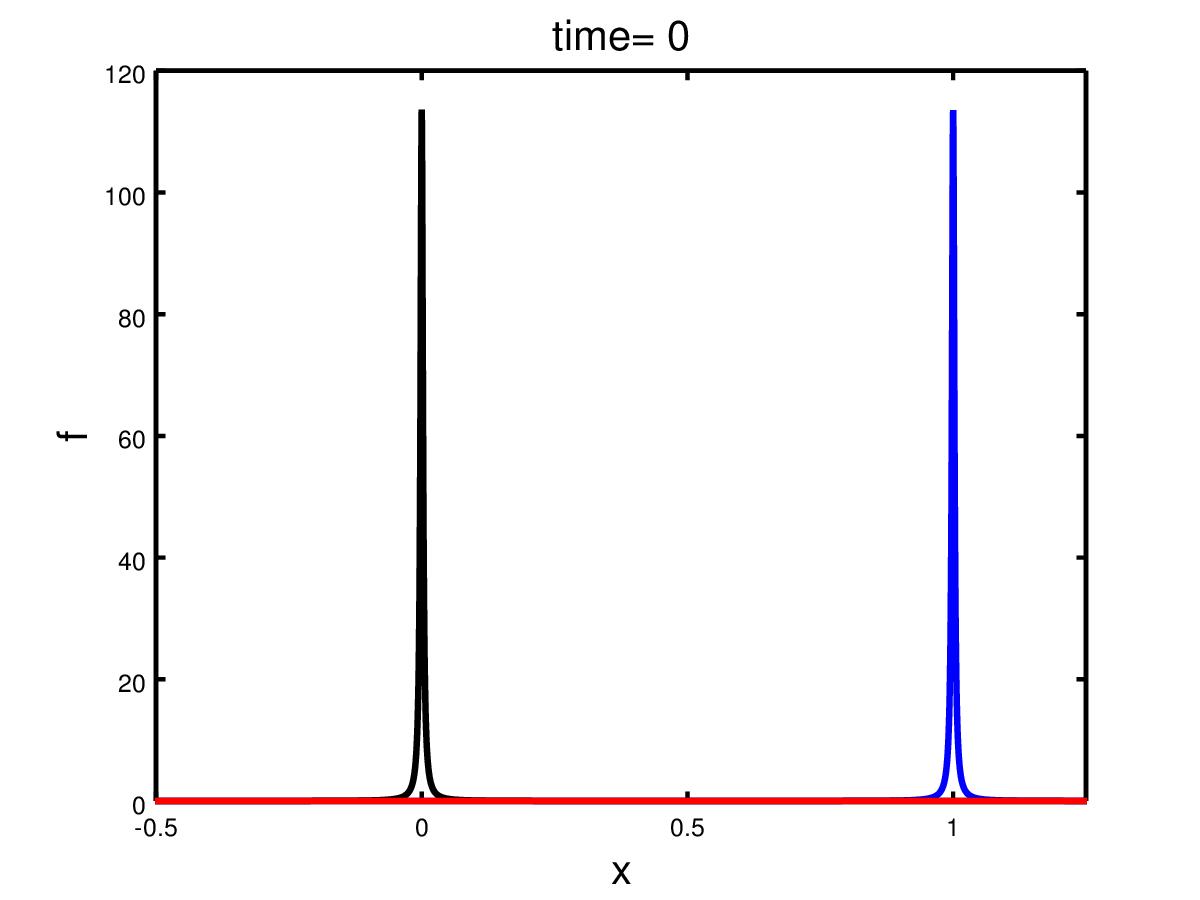}\quad \includegraphics[scale=0.3]{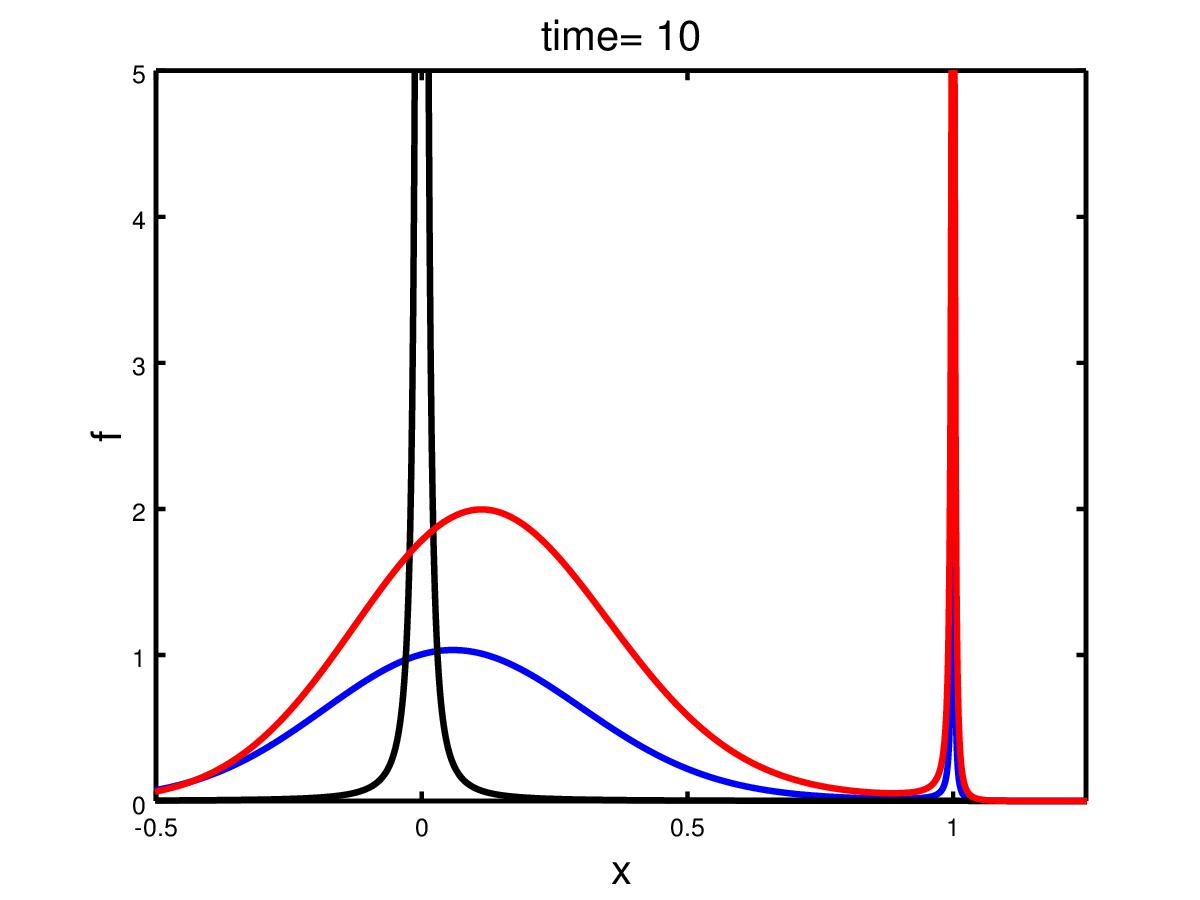}\quad
\includegraphics[scale=0.3]{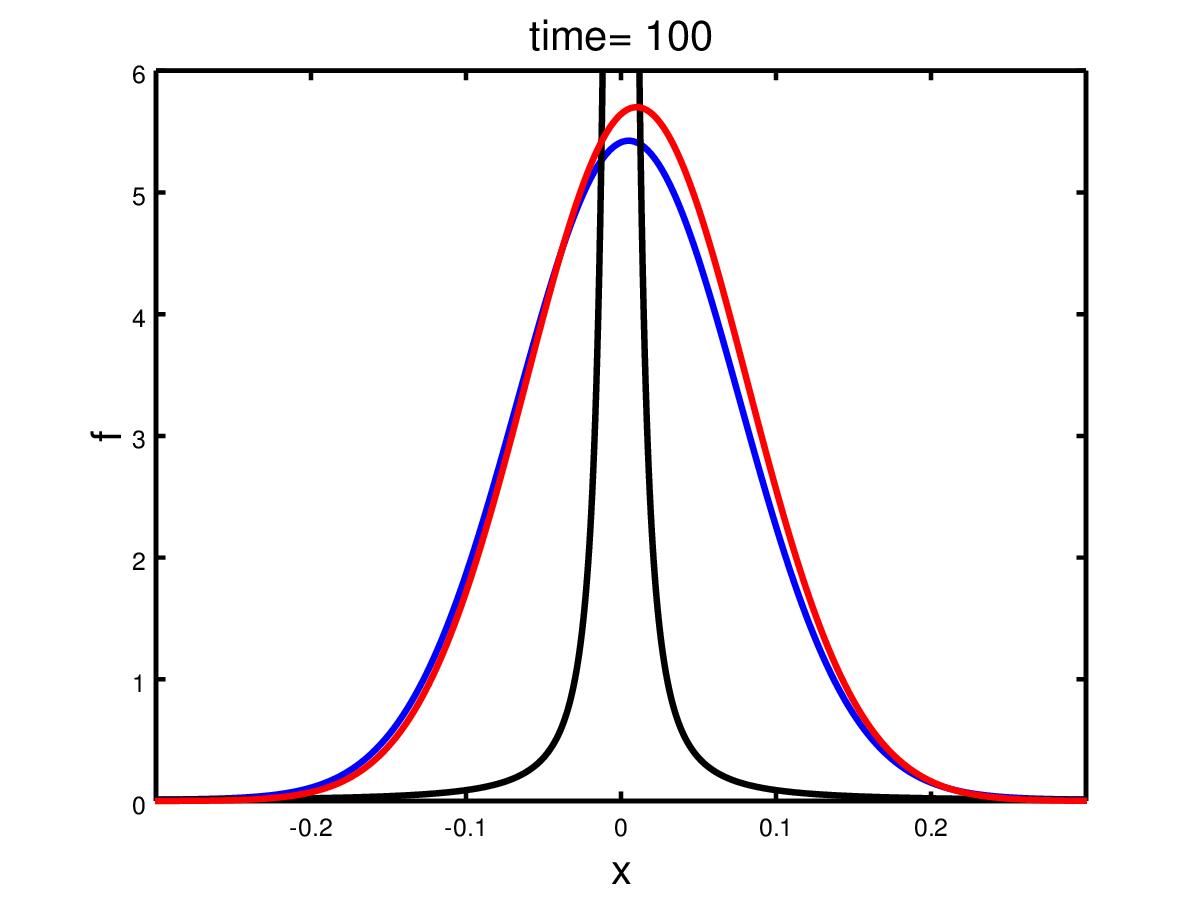}\quad
\includegraphics[scale=0.3]{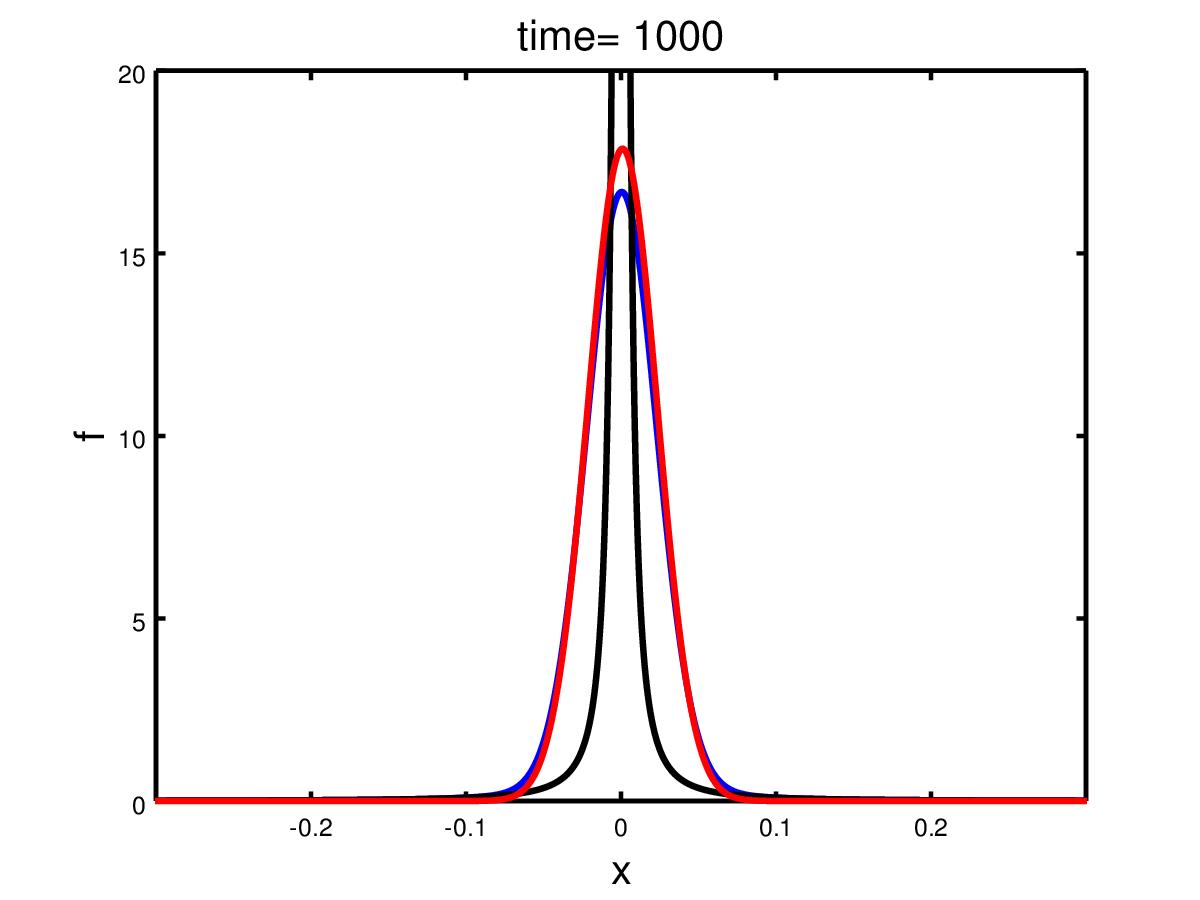}\quad \includegraphics[scale=0.3]{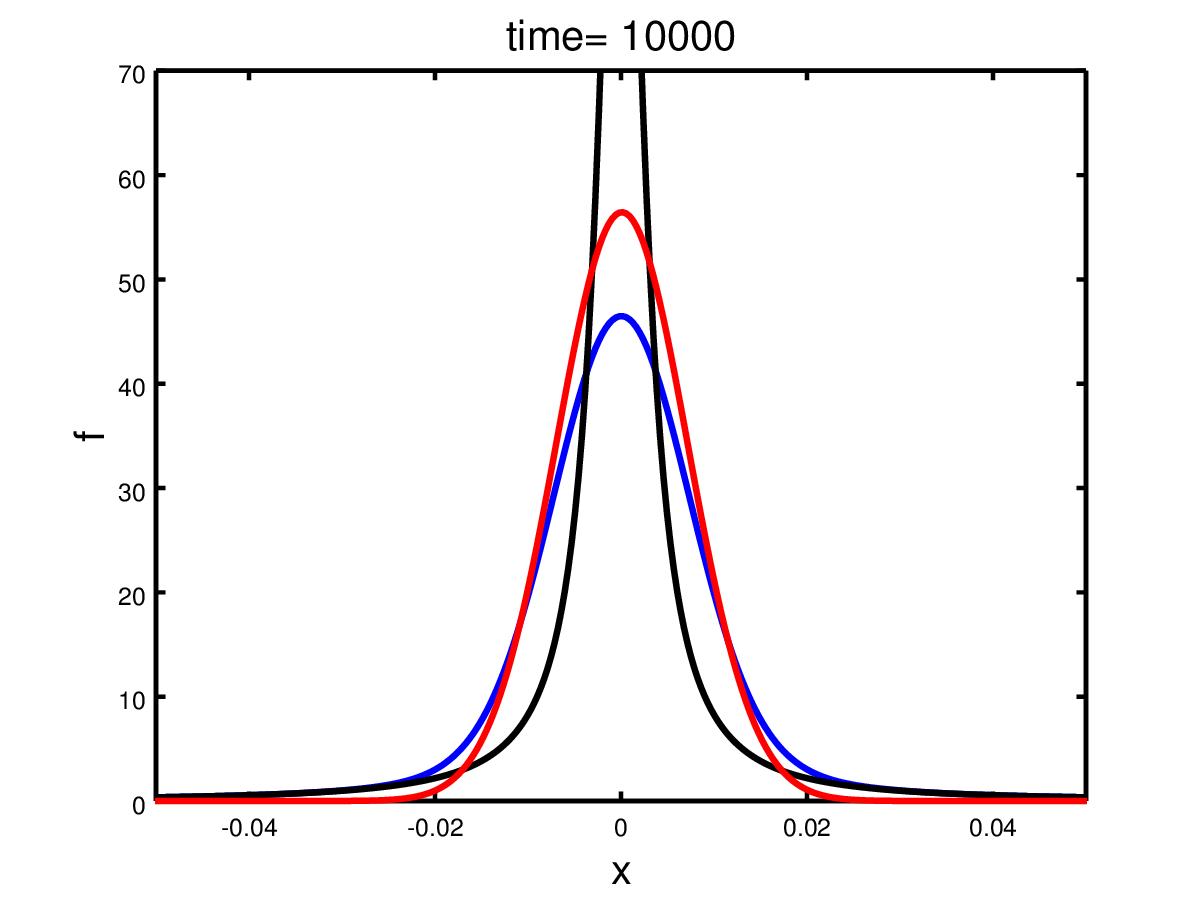}\quad\includegraphics[scale=0.3]{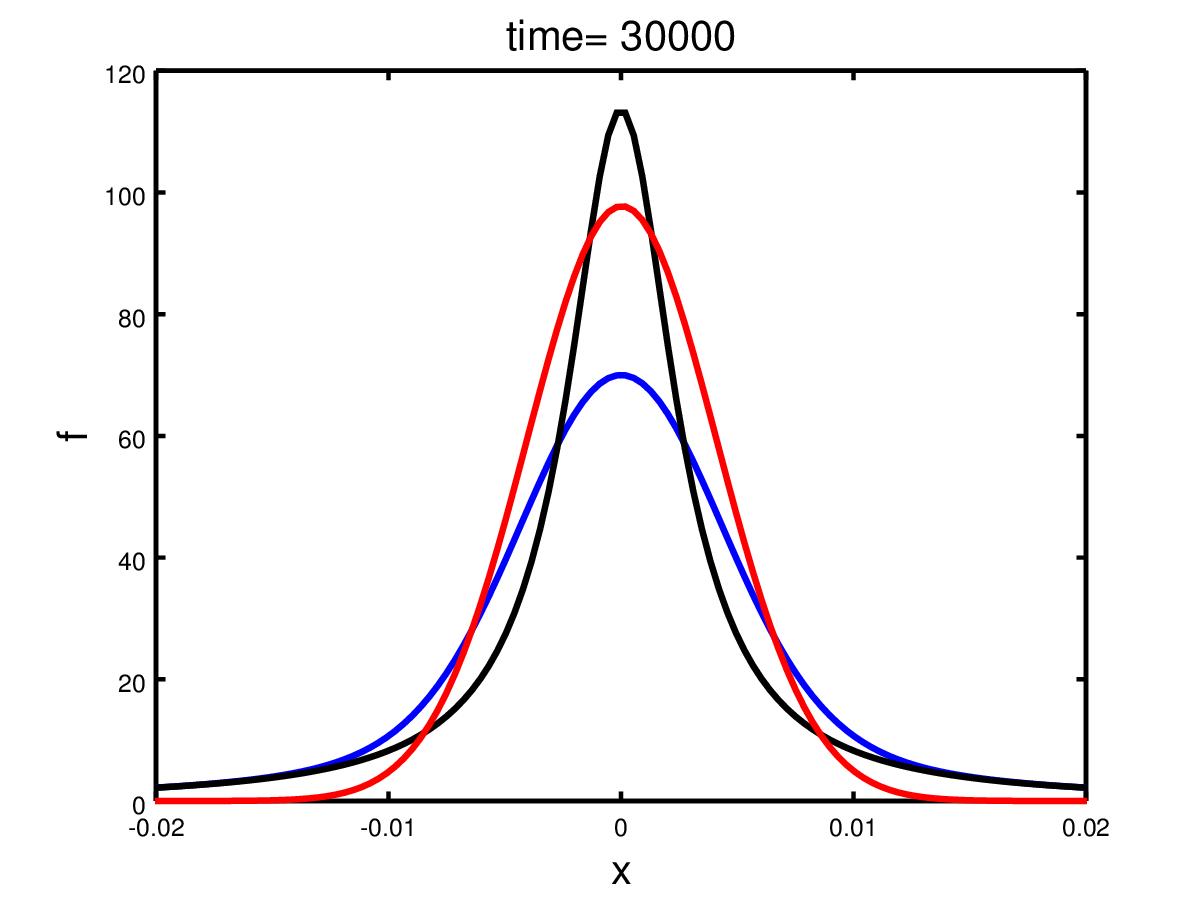}\quad
\includegraphics[scale=0.3]{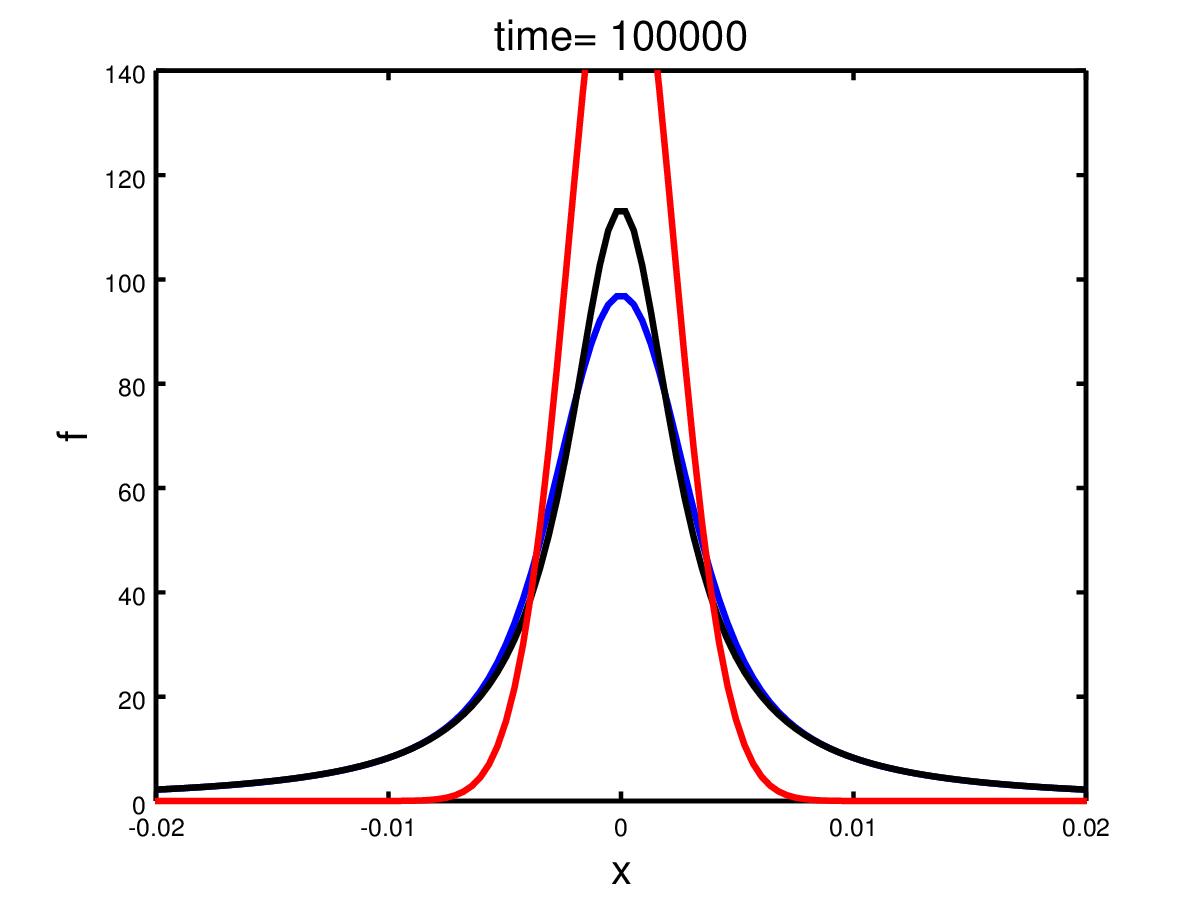}\quad
\includegraphics[scale=0.3]{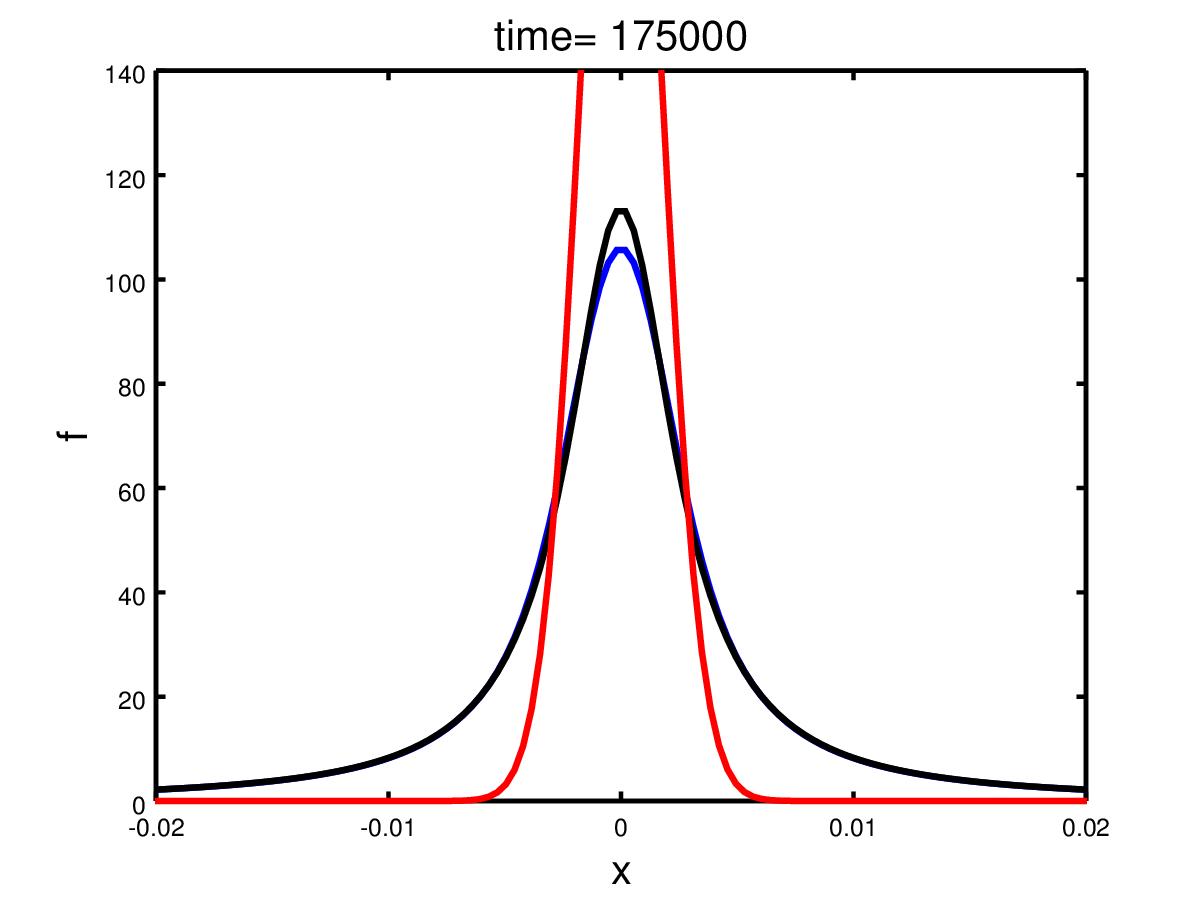}
\caption{The different graphs correspond to different time points, from $t=0$ to $t=175\,000$, 
of the same simulation of \eqref{eqq0} for $\varepsilon=10^{-2}$ (see in the text for a complete description). In each of these plots, the blue (resp. red, black) line represents $x\mapsto f_\var(t,x)$ (resp. $x\mapsto\Gamma_1(t,\var,x)$, $x\mapsto\Gamma_2(\var,x)$). Note that in this figure, the scales of both axis change from one graph to the other, to accommodate with the dynamics of the solution $f(t,\cdot)$.}\label{fig1}
\end{center}
\end{figure}


\begin{figure}[tbp]
\begin{center}
\includegraphics[scale=0.4]{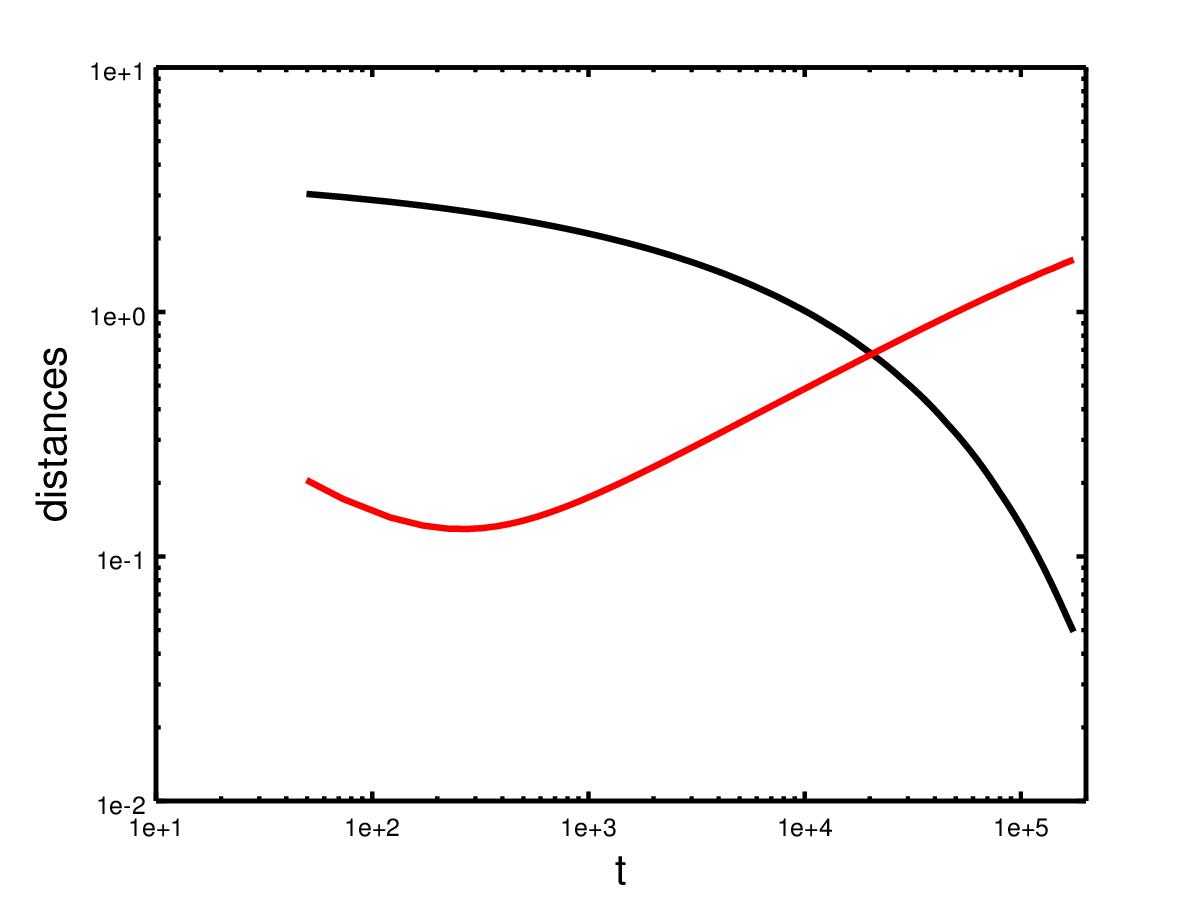}
\caption{Simulation of \eqref{eqq0} with $\varepsilon=10^{-2}$ (see in the text for a complete description). The red line represents $\|f_\var(t,\cdot)-\Gamma_1(t,\var,\cdot)\|_{L^1(I)}$, while the black line represents $\|f_\var(t,\cdot)-\Gamma_2(\var,\cdot)\|_{L^1(I)}$.}\label{fig2}
\end{center}
\end{figure}

 Another way to look at these results is to \gr{consider} $t\geq 0$ and $\var>0$ as two parameters, and to see the approximations presented above as approximations of $f_\var(t,\cdot)$ for some set of parameters: $f_\var(t,\cdot)\sim_{\var\to 0} \Gamma_2(\var,\cdot)$ for $(t,\var)\in\{(\tilde t,\tilde\var);\tilde t\gg \tilde{\var}^{-4}\}$, while $f_\var(t,\cdot)\sim_{\var\to 0} \Gamma_1(t,\var,\cdot)$ for $(t,\var)\in\{(\tilde t,\tilde\var);1\ll\tilde t\ll \tilde{\var}^{-2/3}\}$. We have represented these sets in Fig~\ref{fig3}.

\begin{figure}[tbp]
\begin{center}
\includegraphics[scale=0.4]{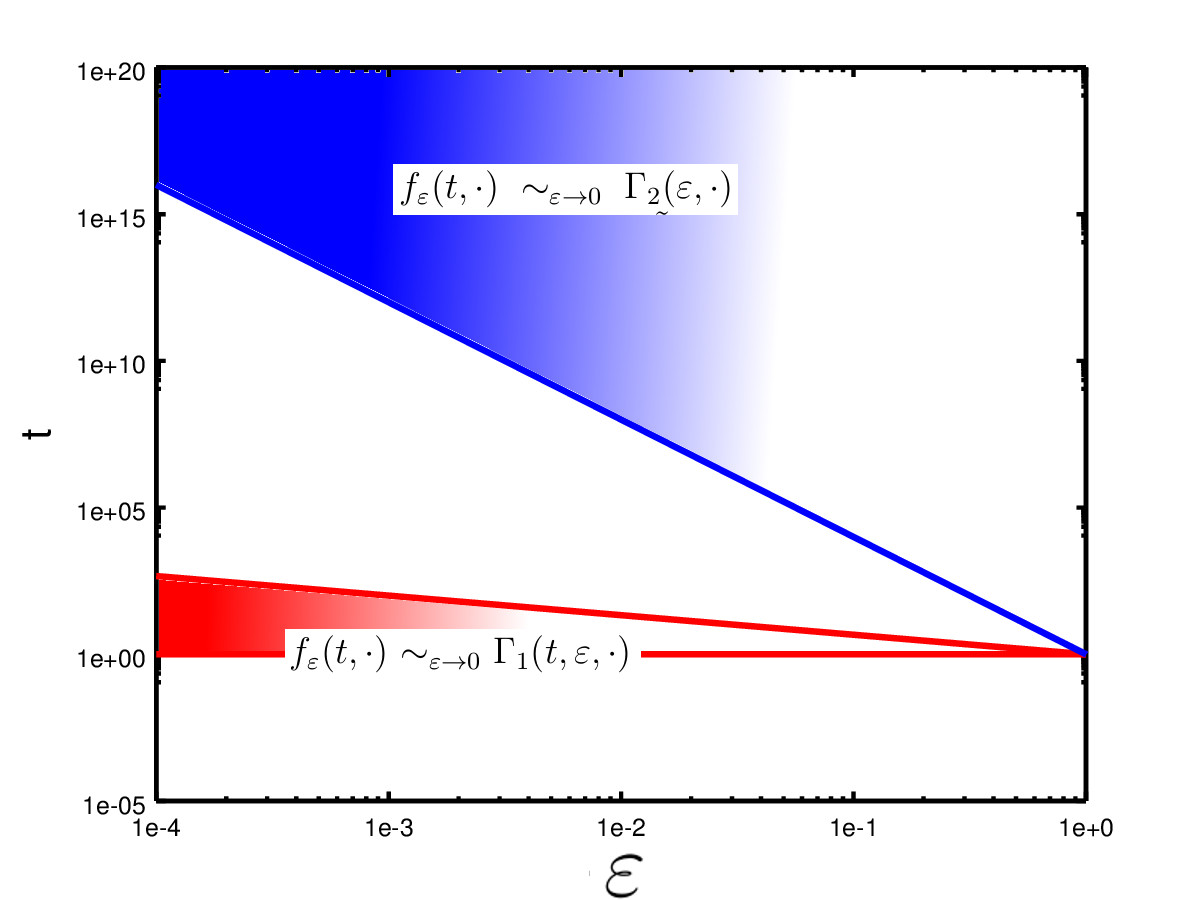}
\caption{Representation of the set $\{(\tilde t,\tilde\var);\tilde t\gg \tilde{\var}^{-4}\}$ (in blue), where the approximation $f_\var(t,\cdot)\sim_{\var\to 0} \Gamma_2(\var,\cdot)$ holds provided that $\var>0$ is small enough; and of the set $\{(\tilde t,\tilde\var);1\ll\tilde t\ll \tilde{\var}^{-2/3}\}$ (in red), where the approximation $f_\var(t,\cdot)\sim_{\var\to 0} \Gamma_1(t,\var,\cdot)$ holds provided that $\var>0$ is small enough.}\label{fig3}
\end{center}
\end{figure}

As described in the Subsection~\ref{subsec:sel-mut-comp}, the phenotypic distribution of species is involved in many ecological and epidemiological problematics. Our study is a general analysis of this problem and we do not have a particular application in mind. An interesting and (to our knowledge) new feature described by our study is that the tails of the traits distribution in a population can change drastically between "short times", that is $1\ll t\ll \var^{-2/3}$ and "large times", that is $t\gg \var^{-4}$: the distribution is initially close to a Gaussian distribution, with small tails, and then converges to a thick tailed Cauchy distribution. \gr{This result could have significant consequences in for \emph{evolutionary rescue}: the tails of the distribution then play an important role. Quantifying the effect of this property of the tails of the distributions would however require further work, in particular on the impact of stochasticity (the number of pathogen is typically large, but finite). The plasticity of the pathogen (see \cite{Chevin}) may also play an important role.}

%


\subsection{Rigorous statements} \label{subsec:rigorous}

Here we state the two main theorems of the paper\gr{, each of them followed by a corollary}. To do so we start by defining the linear operator

$$
 (A_{\varepsilon}f) (x):=(1-\varepsilon)f(x)
-x^{2}\,f(x)+\varepsilon \gamma(x)\,\int_{I}f(y)\mbox{d}y
$$
and denoting by $\lambda_{\varepsilon}$ the dominant eigenvalue of $A_{\varepsilon}$ and by $\psi_{\varepsilon}(x)=\frac{\varepsilon \gamma(x)}{\lambda_{\varepsilon}-(1-\varepsilon)+x^2}$ the corresponding eigenvector (see Proposition \ref{proposition1}).

\begin{theorem}
\label{theorem1} Let us assume that the initial datum $f_0\ge 0$ is  integrable
 on $I$ ($I= ]a,b[$, $-\infty \le a < b \le +\infty$), and $f_0$ is not identically (i.-e. a.e.) $0$.

Then the initial value problem for (\ref{eqq0}) with
$f(0,x)=f_0(x)$ has a unique (global for positive times) mild
solution.
 Moreover, for $\varepsilon>0$ small enough, and any $\rho_{\varepsilon} < (\gamma(0)\pi\varepsilon)^2$,
there exists a constant $C_{\varepsilon}>0$ (depending on
$f_0$  and $\var$) such that
$$ \left\|f(t,\cdot)-\lambda_{\varepsilon}\,\psi_{\varepsilon} \right\|_{L^1(I)} \leq C_{\varepsilon}\,
e^{-\rho_{\varepsilon}\,t}.
$$
Furthermore, taking
$\rho_{\varepsilon}=\frac{\alpha_{\varepsilon}}{2}=\frac{\lambda_{\varepsilon}-(1-\varepsilon)}{2}$, the following
more explicit (in terms of dependence w.r.t $\var$) estimate holds
$$ \left\|f(.,t)-\lambda_{\varepsilon}\,\psi_{\varepsilon} \right\|_{L^1(I)} \leq
K\,
\varepsilon^{\frac{-\hat{K}}{\varepsilon^2}}\, e^{\frac{-\alpha_{\varepsilon}t}{2}},
$$
where $K, \hat{K}>0$ depend on $f_0$ but not on $\var$.

\end{theorem}

\begin{corollary}\label{cor1}
Under the same hypotheses, there exists positive constants $K, \hat{K}$ and $\tilde{K}$ (independent of $\varepsilon$) such that
$$
\left\|f(t,\cdot) - \frac{\varepsilon
\gamma(0)}{(\gamma(0)\pi\varepsilon)^2+x^2}\right\|_{L^{1}(I)} \leq K\,
\varepsilon^{\frac{-\hat{K}}{\varepsilon^2}}\,
e^{-\hat{K}\varepsilon^2 t}+ \tilde{K}\varepsilon
\ln{\left(\frac{1}{\varepsilon}\right)}.
$$
\end{corollary}

\begin{theorem}\label{thm:smallt}
Let $\gamma\in C^0(I)\cup L^\infty(I)$ such that $\int_I \gamma(x)\,dx=1$. Let $f(0,\cdot)\in W^{1,\infty}(I)$ satisfying $f(0,0)>0$ and $\int_I f(0,x)\,dx<1$. There exists a constant $C>0$ such that the solution $f\in C^1(\mathbb R_+\times I)$ of \eqref{eqq0} satisfies
\gr{\begin{equation}\label{est:final-smallt}
\forall t\geq 0,\quad \Bigg\| x \mapsto f(t,x)-\frac{f(0,x)\sqrt t e^{-x^2 t}}{f(0,0)\,\int_I e^{-y^2}\,dy}\Bigg\|_{L^1(I)}\leq 
C \left(\frac 1{\sqrt t}+\varepsilon\, t^{\frac 3 2}\,e^{C\,\varepsilon\, t}
\right).
\end{equation}}
\end{theorem}

\begin{remark}\label{rem:C}
As can be seen from the proof, the constant $C$ appearing in \eqref{est:final-smallt} indeed only depends on some upper bounds on $\|\gamma\|_{L^\infty}$, $\|f(0,\cdot)\|_{W^{1,\infty}}$ and a lower bound on $f(0,0)$, and on $|a-b|$.
\end{remark}

In particular, Theorem~\ref{thm:smallt} implies the following description of the population's phenotypic diversity during transitory times, that is times $t$ satisfying $1\ll t\ll\varepsilon^{-\frac 23}$:

\begin{corollary}\label{cor:smallt}
Let $\gamma\in C^0(I)\cup L^\infty(I)$ such that $\int_I \gamma(x)\,dx=1$. Let $f(0,\cdot)\in W^{1,\infty}(I)$ satisfying $f(0,0)>0$ and $\int_I f(0,x)\,dx<1$.  There exists $C>0$ such that for $\kappa>0$ small enough, as soon as  $\varepsilon<\kappa$, the solution $f\in C^1(\mathbb R_+\times I)$ of \eqref{eqq0} satisfies
\[\forall t\in\left[\kappa^{-2},\kappa^{\frac 23}\varepsilon^{-\frac 23}\right],\quad\Bigg\|f(t,x)-\frac{f(0,x)\sqrt t e^{-x^2 t}}{f(0,0)\,\int_I e^{-y^2}\,dy}\Bigg\|_{L^1}\leq C\kappa.\]
\end{corollary}

These results hold for models which are slightly more general than equation (\ref{eqq0}). In fact, in both theorems one can assume that the competition term is a weighted population instead of the total population number. In Theorem \ref{thm:smallt}, one could also assume that the mutation kernel depends on the parents trait.

\section{Proof of Theorem \ref{theorem1} and of Corollary \ref{cor1}}

We start here the proof of Theorem \ref{theorem1}. We recall that
 $I=]a,b[$, $-\infty \leq a <0< b \leq \infty$, and $\gamma: = \gamma(x)$ is a
bounded, $C^1$ function with bounded derivative, such that $\gamma(x) > 0$ and $\int_{I}
\gamma(x) \,dx=1$. We begin with the study of the linear
operator associated to eq. (\ref{eqq0}).

\subsection{Spectrum of the linear operator}

Let us recall that
\begin{equation}
\label{eq2} (A_{\varepsilon}f) (x):=(1-\varepsilon)f(x)
-x^{2}\,f(x)+\varepsilon \gamma(x)\,\int_{I}f(y)\,dy
 \end{equation}
is the operator corresponding to the linear part in eq. (\ref{eqq0}). It
 acts on functions of the variable $x \in I$.
We begin with a basic lemma which enables to define the semigroup
associated with this operator.
\medskip

\begin{lemma}\label{lemanou} The linear operator $A_{\varepsilon}$, defined
 on $L^1(I)$ and with domain $D(A_{\varepsilon})=\{f \in
L^1(I) : \int_{I} x^{2}\, |f(x)|\, \,dx < \infty \}$, generates
an irreducible positive $C^0$-semigroup (denoted from now on by
$T_{\varepsilon}(t)$).
\end{lemma}

\begin{proof} The multiplication linear operator
$(A_{\varepsilon}^{0}f)(x) := (1-\varepsilon)f(x) -x^{2}\,f(x)$ is
the generator of a positive $C^0$-semigroup.
Since $\gamma$ is strictly positive,
$A_{\varepsilon}-A_{\varepsilon}^0 $ is a positive bounded perturbation whose
only invariant closed ideals are ${0}$ and the whole space $L^1(I)$.
So $T_{\varepsilon}(t)$ is irreducible (see \cite{clement},
Corollary 9.22).
\end{proof}

Next, we present a proposition which gives information about the
spectrum of $A_{\varepsilon}$.
\medskip

\begin{proposition}
\label{proposition1}
 The linear operator $A_{\varepsilon}$ has only one eigenvalue. It is a strictly dominant algebraically simple eigenvalue $\lambda_{\varepsilon}
>1-\varepsilon$ and a pole of the resolvent, with corresponding normalized positive eigenvector $$ \psi_{\varepsilon}(x)=\frac{\varepsilon \gamma(x)}{\lambda_{\varepsilon}-(1-\varepsilon-x^{2})}. $$ Moreover, for $\varepsilon$ small enough, $\lambda_{\varepsilon}<1$.
\par
The rest of the spectrum of the linear operator $A_{\varepsilon}$ is
equal to the interval $J =
[\min(1-\varepsilon-a^2,1-\varepsilon-b^2), 1-\varepsilon].$
\end{proposition}

\begin{proof} In the sequel, the norm $||\,\, ||$ is the $L^1$ norm on $I$.
Let us first show that any $\lambda$ belonging to the set
$J=\text{Range}(1-\varepsilon -x^2)$ belongs to the spectrum of
$A_{\varepsilon}$. In order to do this, for $\lambda=1-\varepsilon
-x_0^2, x_0 \in \mathring{I}$, let us define
$f_n(x)=\frac{n}{2}\left(\chi_{[x_0,x_0+\frac{1}{n}]}(x)-\chi_{[x_0-\frac{1}{n},x_0]}(x)\right)$
for $n$ such that $[x_0-\frac{1}{n},x_0+\frac{1}{n}]\subset I$. We
then have $\|f_n\|=1$ and $\left\|(A_{\varepsilon}-\lambda Id)f_n\right\| =
\frac{n}{2}\int_{x_0-\frac{1}{n}}^{x_0+\frac{1}{n}}|x^2-x_0^2|dx
\rightarrow 0$. So $(\min(1-\varepsilon-a^2,1-\varepsilon-b^2),
1-\varepsilon]$ is contained in the spectrum of $A_{\varepsilon}$.
The claim follows from the fact that the spectrum is a closed set.

On the other hand, notice that (for $x_0 \in I$), $1-\varepsilon
-x_0^2$ is not an eigenvalue, since the potential corresponding
eigenfunction $\frac{\gamma(x)}{x_0^2-x^2}$ is not an integrable
function on $I$ (remember that $\gamma$ does not vanish).

 Let us now compute the resolvent operator of $A_{\varepsilon}$, that is,
 let us try to solve the equation
 \begin{equation}
 \label{eq3}
  A_{\varepsilon}f-\lambda f= g \in L^1(I).
\end{equation}
For $\lambda \notin J$, defining $p:= \int_{I}f(y)\,dy$,
(\ref{eq3}) gives
\begin{equation}
\label{eq4} f(x)=\frac{\varepsilon
\gamma(x)p-g(x)}{\lambda-(1-\varepsilon-x^2)} .
\end{equation}
Integrating, we get
\begin{equation}
 \label{eq5}
\bigg(1-\varepsilon
\int_{I}\frac{\gamma(x)}{\lambda-(1-\varepsilon-x^2)}\,dx
\bigg)\,p
=\int_{I}\frac{-g(x)}{\lambda-(1-\varepsilon-x^2)}\,dx ,
\end{equation}
and $\lambda$ belongs to the resolvent set unless the factor of $p$
on the left hand side vanishes. Therefore $\sigma(A)=J \cup \{
\lambda \in \mathbb{C} :\varepsilon
\int_{I}\frac{\gamma(x)}{\lambda-(1-\varepsilon-x^2)}\,dx=1
\}$.
\par
Since for any real number $\lambda>1-\varepsilon$, the function
$F_{\varepsilon}(\lambda):= \varepsilon
\int_{I}\frac{\gamma(x)}{\lambda-(1-\varepsilon-x^2)}\,dx$ is
continuous, strictly decreasing, and satisfies $\lim_{\lambda \to
1-\varepsilon}F_{\varepsilon}(\lambda)=+\infty$ (recall that
$\gamma(0)>0$) and $\lim_{\lambda \to
+\infty}F_{\varepsilon}(\lambda)=0$, we see that there is a unique
real solution of $F_{\varepsilon}(\lambda)=1$ in
$(1-\varepsilon,\infty)$. We denote it by $\lambda_{\varepsilon}$.
\par
Taking $g(x)=0$ in (\ref{eq3}), we see that $\lambda_{\varepsilon}$
is an eigenvalue with corresponding normalized strictly positive
eigenvector
$$\psi_{\varepsilon}=\frac{\varepsilon
\gamma(x)}{\lambda_{\varepsilon}-\left(1-\varepsilon-x^2\right)}.$$
\par
Taking $g(x)=\psi_{\varepsilon}$ and $\lambda = \lambda_\varepsilon$
we see that the left hand side in (\ref{eq5}) vanishes, whereas the
right hand side is strictly negative, so that
$A_{\varepsilon}f-\lambda_{\varepsilon}f=\psi_{\varepsilon}$ has no
solution and hence $\lambda_{\varepsilon}$ is algebraically simple.\\
\par
Indeed, it also follows from (\ref{eq5}) that the range of
$A_{\varepsilon}-\lambda_{\varepsilon}\, Id$ coincides with the
kernel of the linear form defined on $L^1(I)$ by the $L^\infty$
function $\frac{1}{\lambda_\var-(1-\varepsilon)+x^2}$ (which is the
eigenvector corresponding to the eigenvalue $\lambda_{\varepsilon}$
of the adjoint operator $A_{\varepsilon}^{*}$) and hence it is a
closed subspace of $L^1(I)$. Therefore, $\lambda_{\varepsilon}$ is a
pole of the resolvent (see Theorem A.3.3 of \cite{clement}).
Furthermore, since
$$
F_{\varepsilon}(1) = \varepsilon
\int_{I}\frac{\gamma(x)}{\varepsilon
+x^{2}}\,dx=\int_{I}\frac{\gamma(x)}{1
+\left(\frac{x}{\sqrt{\varepsilon}}\right)^{2}}\,dx
\stackrel{\scriptscriptstyle \varepsilon \rightarrow
0}{\displaystyle \longrightarrow}0,
$$
we see that  $F_{\varepsilon}(1)< 1$ for $\varepsilon$ small enough,
and hence $\lambda_{\varepsilon}<1$.
\par
Substituting $\lambda$ by $a+bi$ in the characteristic equation
\begin{equation}
\label{characteristic} 1+\varepsilon
\int_{I}\frac{\gamma(x)}{(1-\varepsilon-x^2-\lambda)}\,dx=0
\end{equation}
we have that the imaginary part is $-\varepsilon b
\int_{I}\frac{\gamma(x)}{(1-\varepsilon-x^2-\lambda)}\,dx$.
Since $\gamma(x)>0$, there are no non real solutions of
\eqref{characteristic}
\end{proof}
\begin{remark}
Note that $\lim_{\varepsilon \to 0} \lambda_{\varepsilon}=1.$
\end{remark}

We now write an expansion of the eigenvalue $\lambda_{\varepsilon}$.
\medskip

\begin{proposition}
\label{proposition2}
 Let $\lambda_{\varepsilon}$ be the dominant eigenvalue of the operator $A_{\varepsilon}$. Then
$$
\left|\lambda_{\varepsilon}-(1-\varepsilon)-\gamma(0)^2\pi^{2}\varepsilon^{2}\right|=O\left(\varepsilon^3\ln{\frac{1}{\varepsilon}}\right)
$$
\end{proposition}

\begin{proof}
 Let us consider the change of variable $x=\nu_{\varepsilon}z$ where $\nu_{\varepsilon}=\sqrt{\lambda_{\varepsilon}-(1-\varepsilon)}$. We have
$$
1=\varepsilon
\int_{a}^{b}\frac{\gamma(x)}{(\lambda_{\varepsilon}-(1-\varepsilon-x^2))}\,dx=\varepsilon
\int_{\frac{a}{\nu_{\varepsilon}}}^{\frac{b}{\nu_{\varepsilon}}}\frac{\gamma(\nu_{\varepsilon}z)}
{\nu_{\varepsilon}^{2}+(\nu_{\varepsilon}z)^2}
\nu_{\varepsilon}\,dz=\frac{\varepsilon}{\nu_{\varepsilon}}
\int_{\frac{a}{\nu_{\varepsilon}}}^{\frac{b}{\nu_{\varepsilon}}}\frac{\gamma(\nu_{\varepsilon}z)}
{1+z^2}\,dz.
$$
Then
$$
\begin{array}{rcl}
\Big| \frac{\nu_{\varepsilon}}{\varepsilon}- \gamma(0)\pi\Big| &=& \Big|\int_{\frac{a}{\nu_{\varepsilon}}}^{\frac{b}{\nu_{\varepsilon}}}\frac{\gamma(\nu_{\varepsilon}z)}
{1+z^2}\,dz-\gamma(0)\pi\Big|\\ \\
& \leq & \Big|\int_{\frac{a}{\nu_{\varepsilon}}}^{\frac{b}{\nu_{\varepsilon}}}\frac{\gamma(\nu_{\varepsilon}z)}
{1+z^2}\,dz- \int_{\mathbb{R}}\frac{\gamma(\nu_{\varepsilon}z)}
{1+z^2}\,dz\Big|+ \Big|  \int_{\mathbb{R}}\frac{\gamma(\nu_{\varepsilon}z)-\gamma(0))}
{1+z^2}\,dz\Big|\\ \\

& \leq & 4 \|\gamma\|_{\infty}\int_{\frac{B}{\nu_{\varepsilon}}}^{+\infty}\frac{\,dz}{1+z^2}+2 \|\gamma'\|_{\infty}\nu_{\varepsilon}\int_{0}^{\frac{A}{\nu_{\varepsilon}}}\frac{z}{1+z^2}\,dz
\end{array}$$
where we have used
$$
|\gamma(\nu_{\varepsilon}z)-\gamma(0)|\leq \min \left(\|\gamma\|_{\infty},\|\gamma'\|_{\infty} \nu_{\varepsilon}|z|\right)
$$
and have denoted $A:= \frac{\|\gamma\|_{\infty}}{\|\gamma'\|_{\infty}}$ and $B:=\min (|a|, b, A)$.\\
Since
$$
4 \|\gamma\|_{\infty}\int_{\frac{B}{\nu_{\varepsilon}}}^{+\infty}\frac{\,dz}{1+z^2} = 4 \|\gamma\|_{\infty} \arctan{\left(\frac{\nu_{\varepsilon}}{B}\right)} \leq 4 \|\gamma\|_{\infty}\frac{\nu_{\varepsilon}}{B}
$$
and
$$
2 \|\gamma'\|_{\infty}\nu_{\varepsilon}\int_{0}^{\frac{A}{\nu_{\varepsilon}}}\frac{z}{1+z^2}\,dz= \|\gamma'\|_{\infty}\nu_{\varepsilon} \ln{\left(1+\frac{A^2}{\nu_{\varepsilon}^2}\right)}
$$
we obtain
\begin{equation}
\label{ine10}
\Big|\nu_{\varepsilon}-\varepsilon\gamma(0)\pi\Big| \leq \varepsilon \nu_{\varepsilon} \left(\frac{4 \|\gamma\|_{\infty}}{B} + \|\gamma'\|_{\infty}
 \ln{\left(1+\frac{A^2}{\nu_{\varepsilon}^2}\right)} \right)
\end{equation}
which implies
\begin{align*}
&\varepsilon \left( \gamma(0)\pi - \nu_{\varepsilon}\left(\frac{4 \|\gamma\|_{\infty}}{B} + \|\gamma'\|_{\infty}
 \ln{(1+\frac{A^2}{\nu_{\varepsilon}^2})} \right) \right) \\
 &\quad \leq  \nu_{\varepsilon} \leq \varepsilon \left( \gamma(0)\pi + \nu_{\varepsilon}\left(\frac{4 \|\gamma\|_{\infty}}{B} + \|\gamma'\|_{\infty}
 \ln{\left(1+\frac{A^2}{\nu_{\varepsilon}^2}\right)} \right) \right).
\end{align*}
Since
$$
\nu_{\varepsilon}\left(\frac{4 \|\gamma\|_{\infty}}{B} + \|\gamma'\|_{\infty}
 \ln{\left(1+\frac{A^2}{\nu_{\varepsilon}^2}\right)} \right)
\stackrel{\scriptscriptstyle \varepsilon \rightarrow
0}{\displaystyle \longrightarrow}0
$$
we have
\begin{equation}
 \label{ineprop1}
 \frac{\gamma(0)\pi\varepsilon}{2}\leq \nu_{\varepsilon} \leq 2\, \gamma(0)\pi\varepsilon.
\end{equation}
for $\varepsilon$ small enough.\\
Therefore, using \eqref{ineprop1} in \eqref{ine10} we get
\begin{equation}
 \label{ineprop2}
\Big|\nu_{\varepsilon}-\varepsilon\gamma(0)\pi\Big| \leq \varepsilon^2 2 \gamma(0) \pi  \left(\frac{4 \|\gamma\|_{\infty}}{B} + \|\gamma'\|_{\infty}
 \ln{\left(1+\frac{4A^2}{\gamma(0)^2 \pi^{2} \varepsilon^{2}}\right)} \right) \leq C \varepsilon^2 \ln{\left(\frac{1}{\varepsilon}\right)}.
\end{equation}
Finally, by \eqref{ineprop1} and \eqref{ineprop2},

$$
|\lambda_{\varepsilon}-(1-\varepsilon)-\gamma(0)^2\pi^{2}\varepsilon^{2}|=
|\nu_{\varepsilon}+\gamma(0)\pi\varepsilon|\;|\nu_{\varepsilon}-\gamma(0)\pi\varepsilon|\leq 3 \gamma(0) \pi C\varepsilon^3 \ln{\left(\frac{1}{\varepsilon}\right)}.
$$

\end{proof}

\subsection{Asymptotic behavior of the nonlinear equation}

Let us start this subsection with a lemma in which properties of the
spectrum of $\tilde{A}_{\varepsilon}=A_{\varepsilon}-\lambda_{\varepsilon}Id$ are used to study the
asymptotic behavior of the semigroup $\tilde{T}_{\varepsilon}(t)$
generated by $\tilde{A}_{\varepsilon}$.

\begin{lemma}
\label{lemma2.0}
\begin{itemize}
\item[a)] The essential growth bound of the semigroup generated by $\tilde{A}_{\varepsilon}$ is $\omega_{ess}(\tilde{T}_{\varepsilon}) = 1-\varepsilon-\lambda_{\varepsilon}.$
\item[b)] The growth bound of the semigroup generated by $\tilde{A}_{\varepsilon}$ is $\omega_{0}(\tilde{T}_{\varepsilon}) = 0.$
\end{itemize}
\end{lemma}
\begin{proof}
\begin{itemize}
\item[a)] $\tilde{A}_{\varepsilon}$ is a compact (one rank)
perturbation of
$\tilde{A}_{\varepsilon}^0f:=(1-\varepsilon-x^2-\lambda_{\varepsilon})f.$
Then
$\omega_{ess}\left(\tilde{T}_{\varepsilon}\right)=\omega_{ess}\left(\tilde{T}_{\varepsilon}^0\right)$
where $\tilde{T}_{\varepsilon}^0(t)$ is the semigroup generated by
$\tilde{A}_{\varepsilon}^0$ (see \cite{nagel}).

Since $\tilde{A}_{\varepsilon}^0$ is a multiplication operator,
$\omega_{ess}\left(\tilde{T}_{\varepsilon}^0\right) =
1-\varepsilon-\lambda_{\varepsilon}$ and the result follows.

\item[b)] By Proposition \ref{proposition1}, the spectral bound of $\tilde{A}_{\varepsilon}$ is $0$ and
the spectral mapping theorem holds for any positive
$C^{0}-$semigroup on $L^{1}$ (see \cite{clement}).
\end{itemize}
\end{proof}

Let us now write, for a positive non identically zero $f_0$,
$\left(\tilde{T}_{\varepsilon}(t)\right)f_{0}(x)=c_{f_{0}}\psi_{\varepsilon}(x)+v(t,x)$
where $\psi_{\varepsilon}(x)$ is
 the eigenvector corresponding to the eigenvalue 0 of $\tilde{A}_{\varepsilon}$ and
$c_{f_{0}}\psi_{\varepsilon}(x)$ is the spectral projection of $f_0$
on the kernel of $\tilde{A}_{\varepsilon}$ (Note that $c_{f_{0}}>0$
since $f_{0}$ is positive and $\tilde{A}_{\varepsilon}$ is the
generator of an irreducible positive semigroup). We also  define
$\varphi(t) := \int_I v(t,x)\, dx$.
The following lemma gives the asymptotic behavior of $c_{f_0}$:

\begin{lemma}
\label{lemma4} Let us assume that $f_0$ is a positive integrable
function on $I$. Then there exist positive constants $K_1$, $K_2$
(independent of $\varepsilon$ but depending on $f_0$) such that
$K_1\,\varepsilon^2 \leq c_{f_0} \leq K_2.$ Moreover,
$\lim_{\varepsilon \to 0} c_{f_0} = 0.$
\end{lemma}
\begin{proof}
Recall that $c_{f_0}= \langle \psi_{\varepsilon}^{*},f_0 \rangle$
where $\psi_{\varepsilon}^{*}$ is the eigenvector of the adjoint
operator $A_{\varepsilon}^{*}$ corresponding to the eigenvalue
$\lambda_{\varepsilon}$, normalized such that $\langle
\psi_{\varepsilon}^{*},\psi_{\varepsilon} \rangle=1$. Since
$$\psi_{\varepsilon}^{*}= \frac{\gr{\left(\varepsilon
\int_{I}\frac{\gamma(x)}{(\lambda_{\varepsilon}-(1-\varepsilon-x^2))^2}\,dx\right)^{-1}}}{\lambda_{\varepsilon}-(1-\varepsilon-x^2)},$$
we see that
$$
c_{f_0}=
\frac{\int_{I}\frac{f_0(x)}{\lambda_{\varepsilon}-(1-\varepsilon-x^2)}\,dx}{\varepsilon
\int_{I}\frac{\gamma(x)}{(\lambda_{\varepsilon}-(1-\varepsilon-x^2))^2}\,dx}.
$$
Let us start by bounding the denominator from above. Using that, by Proposition \ref{proposition2}, for
$\varepsilon$ small enough, $\lambda_{\varepsilon}-(1-\varepsilon)
\geq \frac{(\gamma(0)\pi\varepsilon)^2}{2},$ we obtain the bound
\begin{equation}
\label{ine1}
\begin{array}{rcl}
\varepsilon
\int_{I}\frac{\gamma(x)}{\left(\lambda_{\varepsilon}-(1-\varepsilon-x^2)\right)^2}\,dx
& \leq & \varepsilon \sup_{x}\gamma(x)
\,\int_{\mathbb{R}}\frac{1}{\left(\frac{(\gamma(0)\pi\varepsilon)^2}{2}+x^2\right)^2}\,dx
\\ \\ & = &\sup_{x}\gamma(x)\, \frac{\sqrt{2}}{\gamma(0)^3(\pi
\varepsilon)^2}=:\frac{K_0}{\varepsilon^2}.
\end{array}
\end{equation}
Similarly, since for $\varepsilon$ small enough,
$\lambda_{\varepsilon}-(1-\varepsilon) \leq
2(\gamma(0)\pi\varepsilon)^2$, so that
\begin{equation}
\label{ine2} \varepsilon
\int_{I}\frac{\gamma(x)}{\left(\lambda_{\varepsilon}-(1-\varepsilon-x^2)\right)^2}\,dx
\geq \varepsilon \min_{[-a,a]}\gamma(x)
\int_{-a}^{a}\frac{\,dx}{\left(2(\gamma(0)\pi\varepsilon)^2+x^2\right)^2}
\geq \frac{K_3}{\varepsilon^2} .
\end{equation}
For the numerator we have, on the one hand,
\begin{equation}
\label{ine3} \varepsilon^{2}
\int_{I}\frac{f_0(x)}{\lambda_{\varepsilon}-(1-\varepsilon-x^2)}\,dx
\leq
\int_{I}\frac{\varepsilon^2}{\frac{(\gamma(0)\pi\varepsilon)^2}{2}+x^2}f_0(x)\,dx,
\end{equation}
where the right hand side tends to $0$ when $\varepsilon$ goes to
$0$ by an easy application of the Lebesgue dominated convergence
theorem (note that the integrand is bounded above by
$\frac{2}{(\gamma(0)\pi)^2}f_0(x)$).

On the other hand, notice that there exists an interval $J \subset I
$ which does not contain $0$ such that $\int_{J}f_0(x)\,dx>0$.
Then, since
$$
\int_{I}\frac{f_0(x)}{\lambda_{\varepsilon}-\left(1-\varepsilon-x^2\right)}\,dx
\geq
\int_{J}\frac{f_0(x)}{\lambda_{\varepsilon}-\left(1-\varepsilon-x^2\right)}\,dx
$$
and
$$
\lim_{\varepsilon \to 0}
\int_{J}\frac{f_0(x)}{\lambda_{\varepsilon}-\left(1-\varepsilon-x^2\right)}\,dx=\int_{J}\frac{f_0(x)}{x^2}\,dx
>0,
$$
there exists a constant $K_4>0$ such that
\begin{equation}
\label{ine4}
\int_{I}\frac{f_0(x)}{\lambda_{\varepsilon}-\left(1-\varepsilon-x^2\right)}\,dx
> K_4.
\end{equation}
By \eqref{ine2} and\eqref{ine3},
$$
c_{f_0}=\frac{\varepsilon^{2}
\int_{I}\frac{f_0(x)}{\lambda_{\varepsilon}-(1-\varepsilon-x^2)}\,dx}{\varepsilon^{3}
\int_{I}\frac{\gamma(x)}{\lambda_{\varepsilon}-(1-\varepsilon-x^2)}\,dx}
\leq \frac{\varepsilon^{2}
\int_{I}\frac{f_0(x)}{\lambda_{\varepsilon}-(1-\varepsilon-x^2)}\,dx}{K_3}
\stackrel{\scriptscriptstyle \varepsilon \rightarrow
0}{\displaystyle \longrightarrow}0
$$
and by \eqref{ine1} and \eqref{ine4}, and $\varepsilon$ small
enough,
$$
c_{f_0} \geq \frac{K_4}{\frac{K_0}{\varepsilon^2}}=:K_1
\varepsilon^2.
$$
This completes the proof.
\end{proof}
\begin{remark}
If $f_0(x)$ is bounded below by a positive number $c$ in a
neighbourhood $(-\delta, \delta)$ of $0$, then the lower estimate
can be improved using that
$$
\int_{-\delta}^{\delta}\frac{\varepsilon}{k^2
\varepsilon^2+x^2}\,dx = \frac{2}{k} \arctan\left(\frac{\delta}{k
\varepsilon}\right) \stackrel{\scriptscriptstyle \varepsilon \rightarrow
0^{+} }{\displaystyle \longrightarrow}\frac{\pi}{k}.
$$
Indeed, for $\varepsilon$ small enough
$$
\begin{array}{rcl}
\varepsilon
\int_{I}\frac{f_0(x)}{\lambda_{\varepsilon}-(1-\varepsilon-x^2)}\,dx
& \geq & \varepsilon \int_{I}\frac{f_0(x)}{2 (\gamma(0) \pi
\varepsilon)^2+x^2}\,dx \\ \\ &  \geq  &c
\int_{-\delta}^{\delta}\frac{\varepsilon}{(\sqrt{2} \gamma(0) \pi)^2
\varepsilon^2+x²}\,dx \stackrel{\scriptscriptstyle \varepsilon
\rightarrow 0^{+}}{\displaystyle \longrightarrow}\frac{c}{\sqrt{2}
\gamma(0)}.
\end{array}
$$
So in this case, for $\varepsilon$  small enough,
$$
c_{f_0} \geq
\frac{\frac{c}{\sqrt{2}\gamma(0)\varepsilon}}{\frac{K_0}{\varepsilon^2}}=:K\varepsilon
$$
for some constant $K$ independent of $\varepsilon$.
\end{remark}

The next two lemmas enable to estimate $\varphi(t)$ (defined above Lemma \ref{lemma4}).
In the first one, the dependence w.r.t. $\var$ is not explicit.

\begin{lemma}
\label{lemma1} For $\varepsilon$ small enough and any
$\rho_{\varepsilon} < (\gamma(0)\pi\varepsilon)^2$ there exists
$K_{\varepsilon}>0$ such that $| \varphi(t)| \leq \|v(t,\cdot)\|
\leq K_{\varepsilon}\,e^{-\rho_{\varepsilon}t}\, \|f_{0}\|.$
\end{lemma}

\begin{proof}
Since by Lemma \ref{lemma2.0}
$\omega_{ess}(\tilde{A}_{\varepsilon})<\omega_{0}(\tilde{A}_{\varepsilon})$,
 we can apply Theorem $9.11$ in \cite{clement}, and get the estimate
$$
\|
v(t,\cdot)\|=\|\tilde{T}_{\varepsilon}(t)f_{0}-c_{f_{0}}\psi_{\varepsilon}\|\leq
K_{\varepsilon}e^{-\eta t}\|f_{0}\| \quad \forall \eta <
\lambda_{\varepsilon}-(1-\varepsilon).
$$
Proposition \ref{proposition2} gives then the statement.
\end{proof}

We now give an estimate of the dependence of $K_{\varepsilon}$ on
$\varepsilon$, provided that $\rho_{\varepsilon}$ is chosen far
enough from its limit value. More precisely, we choose
$\rho_{\varepsilon}=\frac{\lambda_{\varepsilon}-(1-\varepsilon)}{2}=:\frac{\alpha_{\varepsilon}}{2}.$

\begin{lemma}
\label{lem6} For $\varepsilon$ small enough there exists a constant
$K$ independent of $\varepsilon$ and of $f_0$ such that
$$
\left\|\tilde{T}_{\varepsilon}(t)f_{0}-c_{f_{0}}\,\psi_{\varepsilon}\right\|\leq
K \,\varepsilon^{-4} \,e^{\frac{-\alpha_{\varepsilon}}{2}\, t}\,
\|f_0\| .
$$
\end{lemma}

\begin{proof}
Since the proof of this result is quite technical we delay it to the end of this section (subsection \ref{app}).
\end{proof}

We now rewrite equation (\ref{eqq0}) as
\begin{equation}
\label{eq6}
 \frac{\partial f(t,x)}{\partial t}= \tilde{A}_{\varepsilon}f(t,x)+ \bigg(\lambda_{\varepsilon}-\int_{I}f(t,y)\,dy \bigg)\, f(t,x) .
\end{equation}
We look for solutions of (\ref{eq6}) (with positive initial
condition $f_{0} \in L^1(I)$) which can be written as
$f(t,x)=h(t)(\tilde{T}_{\varepsilon}(t)f_{0})(x)$, with $h := h(t)$
a function of time such that $h(0)=1$. Substituting in (\ref{eq6}),
it follows that $f$ is indeed a solution of eq. (\ref{eqq0}) if
$h(t)$ satisfies the following initial value problem for an ordinary
differential equation:
\begin{equation}
 \label{eq7}
h'(t)=\Big(\lambda_{\varepsilon}-h(t)\int_{I}\left(\tilde{T}_{\varepsilon}(t)f_{0}\right)(x)\,dx\Big)h(t),
\qquad h(0)=1,
\end{equation}
or equivalently
\begin{equation}
\label{eq8}
h'(t)=\Big(\lambda_{\varepsilon}-(c_{f_{0}}+\varphi(t))\,
h(t)\Big)\, h(t), \qquad h(0)=1.
\end{equation}

The two next lemmas explain  the asymptotic behavior of $h(t)$. In
the first one, the dependence w.r.t. $\var$ of the constants is not
explicit.

\begin{lemma}
\label{lemma3} For $\varepsilon>0$ small enough and any
$\rho_{\varepsilon} < (\gamma(0)\pi\varepsilon)^2$, there exists a
positive constant $\hat{C}_{\varepsilon}>0$ such that
$\left|h(t)-\frac{\lambda_{\varepsilon}}{c_{f_{0}}}\right| \leq
\hat{C}_{\varepsilon}\, e^{-\rho_{\varepsilon}\, t}.$
\end{lemma}
\begin{proof}
The solution of \eqref{eq8} is explicitly given by
$$
h(t)=\frac{e^{\lambda_{\varepsilon}t}}{1+\int_{0}^{t}(c_{f_{0}}+\varphi(s))\,
e^{\lambda_{\varepsilon}s}\,
\,ds}=\frac{1}{e^{-\lambda_{\varepsilon}t}+\frac{c_{f_0}}{\lambda_{\varepsilon}}\,
(1-e^{-\lambda_{\varepsilon}t})+e^{-\lambda_{\varepsilon}t}\int_{0}^{t}
\varphi(s)\,e^{\lambda_{\varepsilon}s}\,{d}s}.
$$
Then
$$
\begin{array}{rcl}
\left|h(t)-\frac{\lambda_{\varepsilon}}{c_{f_0}}\right|&=&\left|\frac{1}{e^{-\lambda_{\varepsilon}t}+\frac{c_{f_0}}{\lambda_{\varepsilon}}(1-e^{-\lambda_{\varepsilon}t})+e^{-\lambda_{\varepsilon}t}\int_{0}^{t}
\varphi(s)e^{\lambda_{\varepsilon}s}\,{d}s}-\frac{1}{\frac{c_{f_0}}{\lambda_{\varepsilon}}}\right|\\ \\
&=&
\frac{\frac{\lambda_{\varepsilon}}{c_{f_0}}\left|e^{-\lambda_{\varepsilon}t}\big(1-\frac{c_{f_0}}{\lambda_{\varepsilon}}\big)+e^{-\lambda_{\varepsilon}t}\int_{0}^{t}
\varphi(s)e^{\lambda_{\varepsilon}s}\,ds\right|}{e^{-\lambda_{\varepsilon}t}+e^{-\lambda_{\varepsilon}t}\int_{0}^{t}(c_{f_{0}}+\varphi(s))e^{\lambda_{\varepsilon}s}\,{d}s}\\ \\
& \leq &\hat{C}_{\varepsilon}e^{-\rho_{\varepsilon}t},
\end{array}$$
where for the last inequality we have used that the denominator is a
positive continuous function bounded below (it takes the value $1$
for $t=0$ and its limit is $\frac{c_{f_0}}{\lambda_{\varepsilon}}$
when $t$ goes to infinity). We also used the following estimate for
the numerator: since, by Lemma \ref{lemma1}, $|\varphi(s)| \leq
K_{\varepsilon}\,e^{-\rho_{\varepsilon}s}\|f_0\|$, then
$$
\begin{array}{rcl}
\left|e^{-\lambda_{\varepsilon}t}\big(1-\frac{c_{f_0}}{\lambda_{\varepsilon}}\big)+e^{-\lambda_{\varepsilon}t}\int_{0}^{t}
\varphi(s)\,e^{\lambda_{\varepsilon}s}\,ds\right|&\leq&
e^{-\lambda_{\varepsilon}t}\left(\left|1-\frac{c_{f_0}}{\lambda_{\varepsilon}}\right|-\frac{K_{\varepsilon}\,\|f_0\|}{\lambda_{\varepsilon}-\rho_{\varepsilon}}\right)+\frac{K_{\varepsilon}
\,\|f_0\|}{\lambda_{\varepsilon}-\rho_{\varepsilon}}\, e^{-\rho_{\varepsilon}t}\\
\\ & \leq & 2K_{\varepsilon}\, e^{-\rho_{\varepsilon}t}\|f_0\| .
\end{array}
$$

\end{proof}
In order to give an estimate of the dependence of
$\hat{C}_{\varepsilon}$ w.r.t. $\varepsilon,$ we need to bound the
denominator more precisely and to take a value of
$\rho_{\varepsilon}$ separated of its limit value. As in Lemma
\ref{lem6}, we choose
$\rho_{\varepsilon}=\frac{\lambda_{\varepsilon}-(1-\varepsilon)}{2}=:\frac{\alpha_{\varepsilon}}{2}.$

\begin{lemma}
\label{lem5} For $\varepsilon>0$ small enough, there exist
constants $K_{7}$ and $K_{8}$ (independent of $\varepsilon$) such that
$$
\left|h(t)-\frac{\lambda_{\varepsilon}}{c_{f_0}}\right| \leq K_8\,
\varepsilon^{\frac{-K_7}{\varepsilon^2}}\,
e^{-\frac{\alpha_{\varepsilon} t}{2}}.
$$
\end{lemma}

\begin{proof}
Using Lemma \ref{lemma1} and the fact that the second term is
positive we see that
\begin{equation}
\begin{array}{rcl}
e^{-\lambda_{\varepsilon} t} + e^{-\lambda_{\varepsilon} t}
\int_0^t(c_{f_0}+\varphi(s))\,e^{\lambda_{\varepsilon} s} \, ds
&\geq& e^{-\lambda_{\varepsilon} t} + \max \left(0,
c_{f_0}(1-e^{-\lambda_{\varepsilon}
t})-K_{\varepsilon}e^{-\rho_{\varepsilon} t} \right) \\
\\ \geq
e^{-\lambda_{\varepsilon} t_{\varepsilon}}
\end{array}\label{eq666}
\end{equation}
for any $t_{\varepsilon}$ such that
\begin{equation}
\label{eq66} c_{f_0}\,(1-e^{-\lambda_{\varepsilon}
t_{\varepsilon}})-K_{\varepsilon}\, e^{-\rho_{\varepsilon}
t_{\varepsilon}} \geq e^{-\lambda_{\varepsilon} t_{\varepsilon}}.
\end{equation}
(Notice that the left hand side in \eqref{eq66} is an increasing
function of $t_{\varepsilon}$). This indeed happens if
$K_{\varepsilon}\, e^{-\rho_{\varepsilon} t_{\varepsilon}} \leq
\frac{c_{f_0}}{2}$ and $(1+c_{f_0})\,e^{-\lambda_{\varepsilon}
t_{\varepsilon}} \leq \frac{c_{f_0}}{2}.$ Since the second condition
is weaker than the first one for $\varepsilon$ small enough,
\eqref{eq66} holds whenever $t_{\varepsilon}$ is such that
$e^{-\rho_{\varepsilon} t_{\varepsilon}} \leq
\frac{c_{f_0}}{2K_{\varepsilon}}$, i.e., $e^{-\lambda_{\varepsilon}
t_{\varepsilon}} \leq \left( \frac{c_{f_0}}{2 K_{\varepsilon}}
\right)^{\frac{\lambda_{\varepsilon}}{\rho_{\varepsilon}}}$ and
$\varepsilon>0$ is sufficiently small. So, $\left( \frac{c_{f_0}}{2
K_{\varepsilon}}
\right)^{\frac{\lambda_{\varepsilon}}{\rho_{\varepsilon}}}$ is also
a lower bound in \eqref{eq666}, and we finally have
$$
\left|e^{-\lambda_{\varepsilon} t} + e^{-\lambda_{\varepsilon} t}
\int_0^t(c_{f_0}+\varphi(s))e^{-\lambda_{\varepsilon} s} ds\right| \geq
\left( \frac{c_{f_0}}{2 K_{\varepsilon}}
\right)^{\frac{\lambda_{\varepsilon}}{\rho_{\varepsilon}}}.
$$
Using the bound on the numerator given in the proof of Lemma
\ref{lemma3}, the previous estimate and using also Lemma \ref{lem6},
Lemma \ref{lemma4} and Proposition \ref{proposition2}, we obtain
\medskip

\begin{equation}\label{bound1}
\begin{array}{rcl}
\left|h(t)-\frac{\lambda_{\varepsilon}}{c_{f_0}}\right|& \leq &
\frac{2K_{\varepsilon}\, e^{-\rho_{\varepsilon}t}\|f_0\|}{\left(
\frac{c_{f_0}}{2 K_{\varepsilon}}
\right)^{\frac{\lambda_{\varepsilon}}{\rho_{\varepsilon}}}}\\ \\
& \leq &
\frac{2K_{5}\,\varepsilon^{-4}\,e^{-\frac{\alpha_{\varepsilon}t}{2}}\,\|f_0\|}{\left(\frac{K_1
\varepsilon^{2}}{2K_{5}\, \varepsilon^{-4}}\right)^{K_{6}\,
\varepsilon^{-2}}}\\
\\
& = &
2K_{5}\,\left(\frac{2K_5}{K_1}\right)^{-\frac{K_6}{\varepsilon^{2}}}\varepsilon^{-4-
\frac{6K_6}{\varepsilon^2}}\, e^{-\frac{\alpha_{\varepsilon}t}{2}}\,
\|f_0\|
\\
\\ &\leq& K_8 \,\varepsilon^{\frac{-K_7}{\varepsilon^2}}\,
e^{-\frac{\alpha_{\varepsilon}t}{2}}.
\end{array}
\end{equation}
\end{proof}

\medskip

We are now in position to conclude the proof of Theorem
\ref{theorem1}.
\medskip

We recall that $h$ satisfies the integral equation
$$h(t) = 1 +
\int_0^t
\bigg(\lambda_{\varepsilon}-h(s)\,\int_{I}\left(\tilde{T}_{\varepsilon}f_0\right)(x)\,
dx \, \bigg)\, h(s)\, ds
$$
from which the following identity  follows
$$h(t)\tilde{T}_{\varepsilon}(t)f_0 = \tilde{T}_{\varepsilon}(t)f_0 +
\int_0^t
\tilde{T}_{\varepsilon}(t-s)\bigg(\lambda_{\varepsilon}-h(s)\int_{I}\left(\tilde{T}_{\varepsilon}(s)f_0\right)(x)dx\bigg)h(s)\tilde{T}_{\varepsilon}(s)f_0ds,
$$
i.e., $f(x,t)$ is a solution of the variations of constants
equation.\\

On the other hand, the nonlinear part of the right hand side of
(\ref{eq6}) is a locally Lipschitz function of $f \in L^1(I)$. From
this uniqueness follows, whereas global existence is clear from the
previous lemmas.
\medskip

Finally, a standard application of the triangular inequality and
Lemmas \ref{lemma4}, \ref{lemma1} and \ref{lemma3} gives

\begin{equation}
\label{triangle}
\begin{array}{rcl}\left\|f(.,t)-\lambda_{\varepsilon}\psi_{\varepsilon}(x) \right\| &\leq &
\left|h(t)-\frac{\lambda_{\varepsilon}}{c_{f_0}}\right|\;\left\|\tilde{T}_{\varepsilon}(t)f_{0}\right\|
+\frac{\lambda_{\varepsilon}}{c_{f_0}}\left\|\tilde{T}_{\varepsilon}(t)f_{0}-c_{f_0}\,\psi_{\varepsilon}(x)\right\| \\ \\
& \leq & \hat{C}_{\varepsilon}\,e^{-\rho_{\varepsilon}t} \left(K_{2}+K_{\varepsilon}\,e^{-\rho_{\varepsilon}t}\, ||f_0||\right)+\frac{1}{K_1 \varepsilon^2}K_{\varepsilon}e^{-\rho_{\varepsilon}t} \\ \\
 & \leq & C_{\varepsilon} \,e^{-\rho_{\varepsilon}t}.
\end{array}
\end{equation}
Using Lemmas \ref{lem6} and \ref{lem5} in the second inequality of
\eqref{triangle}, the last statement of Theorem \ref{theorem1} follows.

\subsection{Proof of Corollary \ref{cor1}}
By the triangular inequality,
$$
\begin{array}{rcl}
\left\|f(t,\cdot) - \frac{\varepsilon \gamma(0)}{(\gamma(0)\pi\varepsilon)^2+\cdot^2}\right\|_{L^{1}(I)} &\leq& \left\|f(t,\cdot)-\lambda_{\varepsilon}\psi_{\varepsilon}\right\|_{L^{1}(I)}\\ \\ & +& \left\|\lambda_{\varepsilon}\psi_{\varepsilon}(x)-\frac{\varepsilon \gamma(0)}{(\gamma(0)\pi\varepsilon)^2+\cdot^2}\right\|_{L^{1}(I)}
\end{array} .
$$
Hence by Proposition \ref{proposition2} and Theorem \ref{theorem1}, we only need to estimate the last term, for which we have
\begin{align*}
&\left\|\frac{\lambda_{\varepsilon}\varepsilon \gamma}{\lambda_{\varepsilon}-(1-\varepsilon)+\cdot^2}- \frac{\varepsilon \gamma(0)}{(\gamma(0)\pi\varepsilon)^2+\cdot^2}\right\|_{L^{1}(I)} \\
& \quad\leq  \left\|\frac{\varepsilon (\lambda_{\varepsilon}-1)\gamma}{\lambda_{\varepsilon}-(1-\varepsilon)+ \cdot^2}\right\|_{L^{1}(I)}
+ \left\|\frac{\varepsilon \gamma}{\lambda_{\varepsilon}-(1-\varepsilon)+\cdot^2}- \frac{\varepsilon \gamma}{(\gamma(0)\pi\varepsilon)^2+\cdot^2}\right\|_{L^{1}(I)}\\
&\qquad+\left\|\frac{\varepsilon (\gamma-\gamma(0))}{(\gamma(0)\pi\varepsilon)^2+\cdot^2}\right\|_{L^{1}(I)}.
\end{align*}
%
Let us bound the three terms. For the first one we have, by Proposition \ref{proposition2},
$$
\begin{array}{rcl}
\left\|\frac{\varepsilon (\lambda_{\varepsilon}-1)\gamma}{\lambda_{\varepsilon}-(1-\varepsilon)+\cdot^2}\right\|_{L^{1}(I)}& \leq & (\lambda_{\varepsilon}-1)\|\gamma\|_{\infty}\int_{\mathbb{R}}\frac{\varepsilon \,dx}{\frac{(\gamma(0)\pi\varepsilon)^2}{2}+x^2} \\ \\

& = & (\lambda_{\varepsilon}-1)\|\gamma\|_{\infty}\frac{\sqrt{2}}{\gamma(0)}= O(\varepsilon).
\end{array}
$$
For the second one, by Proposition \ref{proposition2} and \eqref{ine1},
\begin{align*}
&\left\|\frac{\varepsilon \gamma}{\lambda_{\varepsilon}-(1-\varepsilon)+\cdot^2}- \frac{\varepsilon \gamma}{(\gamma(0)\pi\varepsilon)^2+\cdot^2}\right\|_{L^{1}(I)}  \\ 
&\quad \leq  |(\gamma(0)\pi\varepsilon)^2-(\lambda_{\varepsilon}-(1-\varepsilon))| \varepsilon \|\gamma\|_{\infty}\int_{\mathbb{R}}\frac{\,dx}{\Big(\frac{(\gamma(0)\pi\varepsilon)^2}{2}+x^2\Big)^{2}}\\ 
&\quad = \left|(\gamma(0)\pi\varepsilon)^2-(\lambda_{\varepsilon}-(1-\varepsilon))\right| \frac{K_0}{\varepsilon^2}=O\left(\varepsilon \ln{\frac{1}{\varepsilon}}\right).
\end{align*}

For the third one, similarly to the proof of Proposition \ref{proposition2}, denoting by $A:= \frac{\|\gamma\|_{\infty}}{\|\gamma'\|_{\infty}}$,
\begin{eqnarray*}
\left\|\frac{\varepsilon (\gamma-\gamma(0))}{(\gamma(0)\pi\varepsilon)^2+\cdot^2}\right\|_{L^{1}(I)}&\leq& 2 \varepsilon \int_{0}^{A} \frac{\|\gamma'\|_{\infty}x}{(\gamma(0)\pi\varepsilon)^2+x^2}\,dx+ 2 \varepsilon\int_{A}^{+\infty} \frac{\|\gamma\|_{\infty}}{(\gamma(0)\pi\varepsilon)^2+x^2}\,dx\\
&=& \varepsilon \|\gamma'\|_{\infty} \ln{(1+\frac{A^2}{(\gamma(0)\pi\varepsilon)^2})}+2 \frac{\|\gamma\|_{\infty}}{\gamma(0)\pi} \arctan{\Big(\frac{\gamma(0)\pi \varepsilon}{A}\Big)}\\
 & = & O\left(\varepsilon \ln{\frac{1}{\varepsilon}}\right).
\end{eqnarray*}

\subsection{Proof of Lemma \ref{lem6} }\label{app}


Let us consider the linear initial value problem
\begin{equation}
\left\{\begin{array}{rcl}
        \frac{\partial u(t,x)}{\partial t}& = & \tilde{A}_{\varepsilon}u(t,x)=(a_{\varepsilon}(x)-\lambda_{\varepsilon})\, u(t,x) +\varepsilon \gamma(x)\,\int_{I}u(t,y)\,dy,\\
\\ u(0,x)&=&u_{0}(x),
       \end{array}\right.
\end{equation}
where $a_{\varepsilon}(x):=1-\varepsilon-x^2$. Let us recall that
$s(\tilde{A}_{\varepsilon})=0$ and $\varepsilon\int_{I}
\frac{\gamma(x)}{\lambda_{\varepsilon}-a_{\varepsilon}(x)}\,dx=1$ (see Proposition \ref{proposition1}).
Applying the Laplace transform with respect to $t$ to the previous
equation, we obtain the identity
$$
\mu \,\mathcal{L}[u](\mu, x)-u_{0}(x)=
(a_{\varepsilon}(x)-\lambda_{\varepsilon})\,\mathcal{L}[u](\mu,x)+\varepsilon\,
\gamma(x)\,\int_{I}\mathcal{L}[u](\mu,y)\, \,dy,
$$
that is
\begin{equation}
\label{eq13}
\mathcal{L}[u](\mu,x)=\frac{u_{0}(x)}{\mu+\lambda_{\varepsilon}-a_{\varepsilon}(x)}+\frac{\varepsilon
\,\gamma(x)}{\mu+\lambda_{\varepsilon}-a_{\varepsilon}(x)}\int_{I}\mathcal{L}[u](\mu,
y)\,dy.
\end{equation}
Integrating (with respect to $x$), we obtain
$$
\int_{I}\mathcal{L}[u](\mu,x)\,dx=\frac{\int_{I}\frac{u_{0}(x)}{\mu+\lambda_{\varepsilon}-a_{\varepsilon}(x)}\,dx}
{1- \int_{I}\frac{\varepsilon
\gamma(x)}{\mu+\lambda_{\varepsilon}-a_{\varepsilon}(x)}\,dx}=\frac{\int_{I}\frac{u_{0}(x)}{\mu+\lambda_{\varepsilon}-a_{\varepsilon}(x)}\,dx}
{\varepsilon \mu
\int_{I}\frac{\gamma(x)}{(\lambda_{\varepsilon}-a_{\varepsilon}(x))(\mu+\lambda_{\varepsilon}-a_{\varepsilon}(x))}\,dx},
$$
where we have used, for the second equality, $\varepsilon\int_{I}
\frac{\gamma(x)}{\lambda_{\varepsilon}-a_{\varepsilon}(x)}=1$.
Substituting in \eqref{eq13},
 we get
\begin{equation}
 \label{eq14}
\mathcal{L}[u](\mu,x)=\frac{u_{0}(x)}{\mu+\lambda_{\varepsilon}-a_{\varepsilon}(x)}+
\frac{\int_{I}\frac{u_{0}(x)}{\mu+\lambda_{\varepsilon}-a_{\varepsilon}(x)}\,dx}
{\mu
\int_{I}\frac{\gamma(x)}{(\lambda_{\varepsilon}-a_{\varepsilon}(x))(\mu+\lambda_{\varepsilon}-a_{\varepsilon}(x))}\,dx}\frac{\gamma(x)}{(\mu+\lambda_{\varepsilon}-a_{\varepsilon}(x))}.
\end{equation}
This Laplace transform is analytic for Re $\mu >0$ (note that
$\lambda_{\varepsilon}-a_{\varepsilon}(x)$ is positive and tends to
zero when $\varepsilon$ tends to zero). Then, for $s>0$, we know, by
the inversion theorem, that
$$
u(t,x)=\frac{1}{2\pi
i}\int_{s-i\infty}^{s+i\infty}\mathcal{L}[u](\mu,x)\,e^{\mu t}\,
\,d\mu.
$$
Using the theorem of residues,  we can shift the integration path to
the left in order to obtain, for any $s' \in
(1-\varepsilon-\lambda_{\varepsilon},0),$
$$
u(t,x)=\text{Res}_{\mu=0} \Big(\mathcal{L}[u](\mu, x)e^{\mu
t}\Big)+\frac{1}{2\pi
i}\int_{s'-i\infty}^{s'+i\infty}\mathcal{L}[u](\mu,x)e^{\mu\, t}\,
\,d\mu,
$$
where
$$
\begin{array}{rcl}
\text{Res}_{\mu=0}\Big(\mathcal{L}[u](\mu, x)e^{\mu
t}\Big)&=& \lim_{\mu \to 0} \mu \mathcal{L}[u](\mu, x)\\ \\
&=& \lim_{\mu \to 0} \left(\frac{\mu
u_{0}(x)}{\mu+\lambda_{\varepsilon}-a_{\varepsilon}(x)}+\frac{
\int_{I}\frac{u_{0}(x)}{\mu+\lambda_{\varepsilon}-a_{\varepsilon}(x)}\,{d}x}
{\int_{I}\frac{\gamma(x)}{(\lambda_{\varepsilon}-a_{\varepsilon}(x))(\mu+\lambda_{\varepsilon}
-a_{\varepsilon}(x))}\,{d}x}\,\frac{\gamma(x)}{\mu+\lambda_{\varepsilon}-a_{\varepsilon}(x)}\right)\\ \\
&=& \frac{\langle u_0, \psi_{\varepsilon}^{*}\rangle}{\langle
\psi_{\varepsilon},
\psi_{\varepsilon}^{*}\rangle}\,\psi_{\varepsilon}(x)=c_{u_0}\,\psi_{\varepsilon}(x)
\end{array}
$$

(let us recall that $\psi_{\varepsilon}(x)=\frac{\varepsilon
\gamma(x)}{\lambda_{\varepsilon}-a_{\varepsilon}(x)}$ and
$\psi_{\varepsilon}^{*}(x)= \frac{\left(\varepsilon
\int_{I}\frac{\gamma(x)\,dx}{(\lambda_{\varepsilon}-(1-\varepsilon-x^2))^2}\right)^{-1}}{\lambda_{\varepsilon}-(1-\varepsilon-x^2)}$).

Thus, we obtain that, for $s' \in
\left(1-\varepsilon-\lambda_{\varepsilon},0\right)$,
\begin{equation}
\label{eq14bis}
 u(t,x)=c_{u_0}\psi_{\varepsilon}(x)+ \frac{1}{2\pi
}\int_{-\infty}^{+\infty}\mathcal{L}[u](s'+i\tau, x)\,
e^{(s'+i\tau)\, t}\, \,d\tau .
\end{equation}
We now define
$g_{\varepsilon}(\mu):=\frac{\int_{I}\frac{u_{0}(x)\,dx}{\mu+\lambda_{\varepsilon}-a_{\varepsilon}(x)}}
{\mu
\int_{I}\frac{\gamma(x)\,dx}{(\lambda_{\varepsilon}-a_{\varepsilon}(x))(\mu+\lambda_{\varepsilon}-a_{\varepsilon}(x))}},
$ so that we can write
\begin{equation}
\label{eq14bisbis}
\begin{array}{rcl}
\frac{1}{2\pi
}\int_{-\infty}^{+\infty}\mathcal{L}[u](s'+i\tau,x)e^{(s'+i\tau) t}
\,d\tau&=& \frac{1}{2\pi}
u_0(x)e^{s't}\int_{-\infty}^{\infty}\frac{e^{i\tau
t}}{s'+\lambda_{\varepsilon}-a_{\varepsilon}(x)+i
\tau}\,d\tau\\ \\ &+&
\frac{1}{2\pi} \gamma(x)e^{s't}\int_{-\infty}^{\infty}\frac{g_{\varepsilon}(s'+i \tau)e^{i\tau t}}{s'+\lambda_{\varepsilon}-a_{\varepsilon}(x)+i \tau}\,d\tau\\ \\
&=&e^{-(\lambda_{\varepsilon}-a_{\varepsilon}(x))t}u_{0}(x)\\ \\
&+&\frac{1}{2\pi}
\gamma(x)e^{s't}\int_{-\infty}^{\infty}\frac{g_{\varepsilon}(s'+i
\tau)e^{i\tau t}}{s'+\lambda_{\varepsilon}-a_{\varepsilon}(x)+i
\tau}\,d\tau,
\end{array}
\end{equation}
where we used the estimate
$s'+\lambda_{\varepsilon}-a_{\varepsilon}(x)>0$ and the identity
$\int_{-\infty}^{\infty}\frac{e^{i \tau t}}{\alpha + i
\tau}\,d\tau=2 \pi e^{-\alpha t}$ (for $\alpha >0$).
\par
We now would like to find a bound for $\left\|\frac{1}{2\pi}
\gamma(x)e^{s't}\int_{-\infty}^{\infty}\frac{g_{\varepsilon}(s'+i
\tau)e^{i\tau t}}{s'+\lambda_{\varepsilon}-a_{\varepsilon}(x)+i
\tau}\,d\tau\right\|_{\infty}$. \\We see that
\begin{equation}
\label{eq15}
\begin{array}{rcl}
 \left\|
\gamma(x)\,e^{s't}\int_{-\infty}^{\infty}\frac{g_{\varepsilon}(s'+i
\tau)e^{i\tau t}}{s'+\lambda_{\varepsilon}-a_{\varepsilon}(x)+i
\tau}\,d\tau\right\|_{\infty}&\leq& e^{s't}
\|\gamma\|_{\infty}\sup_{x}\Big|\int_{-\infty}^{\infty}\frac{g_{\varepsilon}(s'+i
\tau)e^{i\tau t}}{s'+\lambda_{\varepsilon}-a_{\varepsilon}(x)+i
\tau}\,d\tau \Big|\\ \\ &=& e^{s't}
\sup_{x}\Big|\int_{-\infty}^{\infty}\frac{g_{\varepsilon}(s'+i
\tau)e^{i\tau t}}{s'+\lambda_{\varepsilon}-a_{\varepsilon}(x)+i
\tau}\,d\tau \Big|
\end{array}
\end{equation}
and
\begin{equation}
\label{eq16}
\begin{array}{rcl}
\Big|\int_{-\infty}^{\infty}\frac{g_{\varepsilon}(s'+i \tau)e^{i\tau
t}}{s'+\lambda_{\varepsilon}-a_{\varepsilon}(x)+i \tau}\,d\tau
\Big| &\leq & \int_{-\infty}^{\infty}\frac{|g_{\varepsilon}(s'+i
\tau)|}{|s'+\lambda_{\varepsilon}-a_{\varepsilon}(x)+i
\tau|}\,d\tau\\ \\ & \leq &
\int_{-\infty}^{\infty}\frac{|g_{\varepsilon}(s'+i
\tau)|}{|s'+\lambda_{\varepsilon}-(1-\varepsilon)+i
\tau|}\,d\tau
\end{array}
\end{equation}
since $|s'+\lambda_{\varepsilon}-a_{\varepsilon}(x)+i \tau| \geq
|s'+\lambda_{\varepsilon}-(1-\varepsilon)+i \tau|$.
\par
Let us then find an upper bound for $g_{\varepsilon}(s'+i\tau)$. For
the numerator of $g_{\varepsilon}(s'+i\tau)$ we can estimate
$$
\left|\int_{I}\frac{u_{0}(x)}{s'+i\tau+\lambda_{\varepsilon}-a_{\varepsilon}(x)}\,dx\right|
\leq \frac{\|u_{0}\|_{1}}{|s'+i \tau +
\lambda_{\varepsilon}-(1-\varepsilon)|}.
$$
We now find a lower bound for the denominator of
$g_{\varepsilon}(s'+i\tau)$. We use the elementary estimate  $|z|
\geq \max (|\text{Re}z|, |\text{Im}z|) $ and we start with the real
part.
$$
\begin{array}{rcl}
 \Big|\text{Re}\int_{I}\frac{\gamma(x)}{(\lambda_{\varepsilon}-a_{\varepsilon}(x))(s'+i\tau+\lambda_{\varepsilon}-a_{\varepsilon}(x))}\,dx\Big|&=&
 \Big|\int_{I}\frac{\gamma(x)}{(\lambda_{\varepsilon}-a_{\varepsilon}(x))}\frac{s'+\lambda_{\varepsilon}-a_{\varepsilon}(x)}{|s'+i\tau+\lambda_{\varepsilon}-a_{\varepsilon}(x)|^2}\,dx\Big| \\ \\&=&
\int_{I}\frac{\gamma(x)}{(\lambda_{\varepsilon}-a_{\varepsilon}(x))}\frac{s'+\lambda_{\varepsilon}-a_{\varepsilon}(x)}{|s'+i\tau+\lambda_{\varepsilon}-a_{\varepsilon}(x)|^2}\,dx
\\ \\&=&
\int_{I}\frac{\gamma(x)}{(\lambda_{\varepsilon}-a_{\varepsilon}(x))\big(s'+\lambda_{\varepsilon}-a_{\varepsilon}(x)+\frac{\tau^2}{s'+\lambda_{\varepsilon}-a_{\varepsilon}(x)}\big)}\,dx\\ \\
& \geq &
\int_{I}\frac{\gamma(x)}{(\lambda_{\varepsilon_{0}}-(1-{\varepsilon_{0}})+x^2)\big(\lambda_{\varepsilon_{0}}-(1-{\varepsilon_{0}})+x^2+\frac{\tau^2}{x^2}\big)}\,dx
\\ & = &
\int_{I}\frac{x^2\gamma(x)}{(\lambda_{\varepsilon_{0}}-(1-{\varepsilon_{0}})+x^2)((\lambda_{\varepsilon_{0}}-(1-{\varepsilon_{0}})+x^2)x^2+\tau^2)}\,dx\\
\\ & =:&F(\tau),
\end{array}
$$
where in the last inequality we used the estimates
$s'+\lambda_{\varepsilon}-a_{\varepsilon}(x)<
\lambda_{\varepsilon}-a_{\varepsilon}(x)$,
$s'+\lambda_{\varepsilon}-(1-\varepsilon)>0$. We also used that,
since $\lambda_{\varepsilon}-(1-\varepsilon)$ is strictly positive
and tends to zero when $\varepsilon$ goes to zero,
there exists $\varepsilon_0$ such that $\forall \varepsilon < \varepsilon_0$ we have $\lambda_{\varepsilon_{0}}-(1-{\varepsilon_{0}})> \lambda_{\varepsilon}-(1-{\varepsilon})$.\\
In a similar way, for the imaginary part,
$$
\begin{array}{rcl}
 \Big|\text{Im}\int_{I}\frac{\gamma(x)}{(\lambda_{\varepsilon}-a_{\varepsilon}(x))(s'+i\tau+\lambda_{\varepsilon}-a_{\varepsilon}(x))}\,dx\Big|&=&
\Big|\int_{I}\frac{\gamma(x)}{(\lambda_{\varepsilon}-a_{\varepsilon}(x))}\frac{-\tau}{(s'+\lambda_{\varepsilon}-a_{\varepsilon}(x))^2+
\tau^2}\,dx\Big|\\ \\ &=& | \tau |
\int_{I}\frac{\gamma(x)}{(\lambda_{\varepsilon}-a_{\varepsilon}(x))\big((s'+\lambda_{\varepsilon}-a_{\varepsilon}(x))^2+
\tau^2\big)}\,dx \\ \\ & \geq & | \tau |
\int_{I}\frac{\gamma(x)}{(\lambda_{\varepsilon_0}-(1-\varepsilon_0)+x^2)\big((\lambda_{\varepsilon_0}-(1-\varepsilon_0)+x^2)^2+
\tau^2\big)}\,dx \\ \\ &=:& G(\tau).
\end{array}
$$
Defining $H(\tau):=\max(F(\tau),G(\tau))$ we see that
\begin{equation}
 \label{eq17}
|g_{\varepsilon}(s'+i \tau)| \leq \frac{\frac{\|u_0\|_{1}}{|s'+i \tau
+\lambda_{\varepsilon}-(1-\varepsilon)|}}{|s'+i \tau| H(\tau)},
\end{equation}
and then, using \eqref{eq15}, \eqref{eq16} and \eqref{eq17}
\begin{align*}
&\left|\left|
\gamma(x)e^{s't}\int_{-\infty}^{+\infty}\frac{g_{\varepsilon}(s'+i
\tau)e^{i\tau t}}{s'+\lambda_{\varepsilon}-a_{\varepsilon}(x)+i
\tau}\,d\tau\right|\right|_{\infty} \\ 
&\quad \leq e^{s't}\int_{-
\infty}^{+\infty} \frac{\,d\tau}{\sqrt{s'^{2}+ \tau^{2}}|s'+i
\tau + \lambda_{\varepsilon}-(1-\varepsilon)|^2 H(\tau)}\|u_0\|_{1}.
\end{align*}

Now, since $F$ and $G$ are strictly positive continuous functions,
$F(0)>0$ and $\tau G(\tau)$ tends to a positive limit when $\tau$
goes to $\infty$, there exists a  constant $C>0$ (independent of
$\varepsilon$) such that $H(\tau) \geq \frac{C}{1+\tau}$. Choosing
$s'= -\frac{\alpha_{\varepsilon}}{2}$, where $\alpha_{\varepsilon} =
\lambda_{\varepsilon}-(1-\varepsilon),$ we can write

\begin{eqnarray*}
\|\gamma(x)\,e^{s't}\,\int_{-\infty}^{+\infty}\frac{g_{\varepsilon}(s'+i
\tau)e^{i\tau t}}{s'+\lambda_{\varepsilon}-a_{\varepsilon}(x)+i
\tau}\,d\tau\|_{\infty}  &\leq& \frac{e^{-
\alpha_{\varepsilon}t}}{C}\int_{0}^{+\infty}\frac{2(1+\tau)}{\big((\frac{\alpha_{\varepsilon}}{2})^2+\tau^2\big)^\frac{3}{2}}\,d\tau\|u_{0}\|_{1}\\
&=&
\frac{e^{-\frac{\alpha_{\varepsilon}t}{2}}}{C}\left(\frac{8}{\alpha_{\varepsilon}^2}+\frac{4}{\alpha_{\varepsilon}}\right)
 \|u_{0}\|_1.
\end{eqnarray*}


Finally, going back to \eqref{eq14bis} and using \eqref{eq14bisbis},
we end up with
$$
\|u((, \cdot)-c_{u_{0}}\psi_{\varepsilon}\| \leq
\left(1+\frac{1}{\pi
C}\left(\frac{4}{\alpha_{\varepsilon}^2}+\frac{2}{\alpha_{\varepsilon}}\right)
\right)\, e^{-\frac{\alpha_{\varepsilon}t}{2}}\,\|u_{0}\|_1\leq
K_5\,\varepsilon^{-4}e^{-\frac{\alpha_{\varepsilon}t}{2}}\,\|u_{0}\|_1.
$$

\section{Proof of Theorem~\ref{thm:smallt}}

We start here the proof of Theorem~\ref{thm:smallt}. From now on, $C$ will  designate a
strictly  positive constant depending only on some upper bounds on $\|\gamma\|_{L^\infty}$, $\|f(0,\cdot)\|_{W^{1,\infty}}$, a lower bound on $f(0,0)$ (see Remark~\ref{rem:C}), and on $|b-a|$.
\medskip

 Thanks to the variation of the constant formula, the solution $f$ of (\ref{eqq0}) satisfies:
\begin{eqnarray}
 f(t,x)&=&f(0,x)\, e^{(1-\varepsilon-x^2)\,t-\int_0^t\int_I f(s,y)\,dy\,ds}\nonumber\\
&&+\, \varepsilon \int_0^t\left(\gamma(x)\int_I f(s,y)\,dy\right)e^{(1-\varepsilon-x^2)(t-s)-\int_s^t\int_I f(\sigma,y)\,dy\,d\sigma}\,ds\nonumber\\
&=&f(0,x)\, e^{(1-\varepsilon-x^2)\,t-\int_0^t\gr{\mathcal I}(s)\,ds}\nonumber\\
&&+\,\varepsilon \int_0^t\gr{\left(\gamma(x)\, \mathcal I(s) \right)} e^{(1-\varepsilon-x^2)(t-s)-\int_s^t {\mathcal I}(\sigma)\,d\sigma}\,ds,\label{eq:varconst}
\end{eqnarray}
where
\begin{equation*}
\mathcal I(t):=\int_I f(t,y)\,dy.
\end{equation*}
Obtaining a precise estimate on $t\mapsto e^{(1-\varepsilon)(t-s)-\int_s^t\mathcal I(\sigma)\,d\sigma}$ is the key to prove Theorem~\ref{thm:smallt}.

\subsection{Preliminary estimates}

If we sum \eqref{eq:varconst} along $x\in\mathbb R$, we get, for $t\geq0$:
\begin{eqnarray}
 \mathcal I(t)&=&\left(\int_I f(0,x)\, e^{-x^2t}\,dx\right) e^{(1-\varepsilon)t-\int_0^t\mathcal I(s)\,ds}\nonumber\\
&&+\,\varepsilon \int_0^t\left(\int_I\int_I\gamma(x) f(s,y)e^{-x^2(t-s)}\,dx\,dy\right)e^{(1-\varepsilon)(t-s)-\int_s^t\mathcal I(\sigma)\,d\sigma}\,ds\nonumber\\
&=&\frac{z_1(t)}{\sqrt{t}} e^{(1-\varepsilon)t-\int_0^t\mathcal I(s)\,ds}+\varepsilon \int_0^t\frac{z_2(s,t-s)}{\sqrt{t-s}}e^{(1-\varepsilon)(t-s)-\int_s^t\mathcal I(\sigma)\,d\sigma}\,ds,\label{eq:varconst2}
\end{eqnarray}
where
\[z_1(t):=\sqrt{t} \int_I f(0,x)\,e^{-x^2t}\,dx,\quad z_2(\sigma,\tau)=\sqrt{\tau}\int_I\int_I \gamma(x) \,f(\sigma,y)\, e^{-x^2\tau}\,dx\,dy.\]

If we differentiate $\mathcal I$ with respect to $t$, we get
\begin{eqnarray*}
 \frac{\partial \mathcal I}{\partial t}(t)&=&\mathcal I(t)\left(1-\varepsilon-\mathcal I(t)\right)-\int_I x^2f(t,x)\,dx+\varepsilon\int_I\int_I \gamma(x)f(t,y)\,dx\,dy\\
&\leq&\mathcal I(t)\left(1-\varepsilon -\mathcal I(t)\right)+\varepsilon\, \mathcal I(t)\\
&\leq&\mathcal I(t)\left(1-\mathcal I(t)\right),
\end{eqnarray*}
which implies, since $\mathcal I(0)\leq 1$, that
\begin{equation}\label{alphabound}
 0\leq \mathcal I(t)\leq 1.
\end{equation}
Thanks to \eqref{eq:varconst2}, \eqref{alphabound} and the nonnegativity of $z_1,\,z_2$, one gets
\begin{equation}\label{estz1}
 \frac{z_1(t)}{\sqrt{t}} e^{(1-\varepsilon)t-\int_0^t\mathcal I(s)\,ds}\leq C,
\end{equation}
while for some constants $C,\,C'>0$,
\begin{eqnarray*}
 z_1(t)&=&\int_I  f\left(0,\frac x {\sqrt{t}}\right)e^{-x^2}\,dx\\
&\geq&\frac 1C\int_{-C'}^{C'}f\left(0,\frac x {\sqrt{t}}\right)\,dx\geq \frac 1{C},
\end{eqnarray*}
for $t\geq 1$. Note that here we used a lower bound on $f(0,\cdot)$ around $x=0$ (\gr{we have assumed} that $f(0,0)>0$ and that  $f(0,\cdot)$ is continuous). \gr{Thanks to this lower bound,} \eqref{estz1} becomes
\begin{equation}\label{estexpalpha}
 e^{(1-\varepsilon)t-\int_0^t\mathcal I(s)\,ds}\leq C\,\sqrt t .
\end{equation}
Thanks to \eqref{estexpalpha} and \eqref{alphabound}, we can estimate the second term of \eqref{eq:varconst2} as follows:
\begin{eqnarray}
 w(t)&:=&\varepsilon \int_0^t\frac{z_2(s,t-s)}{\sqrt{t-s}}\gr{e^{(1-\varepsilon)\,t-\int_0^t\mathcal I(\sigma)\,d\sigma}e^{\varepsilon s+\int_0^s\left(\mathcal I(\sigma)-1\right)\,d\sigma}}\,ds\nonumber\\
&\leq&C\,\varepsilon\,\sqrt t\, \|z_2\|_{L^\infty(I)}
\int_0^t\frac{e^{C\varepsilon s}}{\sqrt{t-s}}\,ds\nonumber\\
&\leq&C\,\varepsilon\,\sqrt t\,\|z_2\|_{L^\infty(I)}\, e^{C\varepsilon t} \int_0^t\frac{e^{\gr{-}C\varepsilon s}}{\sqrt{s}}\,ds\leq C\, \varepsilon t\, \|z_2\|_{L^\infty(I)}\, e^{C\varepsilon t}.\label{def:w}
\end{eqnarray}
In order to estimate $\|z_2\|_{L^\infty(I)}$, we proceed as follows:
\begin{eqnarray*}
 z_2(s,\tau)&=&\sqrt{\tau}\int_I\int_I \gamma\left(\frac x{\sqrt \tau}\right)\, f(s,y)\,e^{-x^2\tau}\,dx\,dy\\
&\leq&\frac {C\,\mathcal I(s)}{\sqrt \tau}\,\int_I e^{-x^2\,\tau}\,dx \leq C.
\end{eqnarray*}
This estimate combined with \eqref{def:w} implies that $w(t)\geq 0$ satisfies
\begin{equation}\label{eq:est4}
\gr{w(t)\leq C\,\varepsilon \,t \,e^{C\,\varepsilon\, t},}
\end{equation}
Since $f(0,\cdot)\in W^{1,\infty}(I)$, we can estimate
\begin{eqnarray}
 z_1(t)&=&\int_I \left(f(0,0)+\int_0^{\frac x{\sqrt t}}\frac{\partial f}{\partial x}\left(0,z\right)\,dz\right)
\,e^{-x^2}\,dx\nonumber\\
&=&f(0,0)\, \int_I e^{-x^2}\,dx  +\lambda(t),\label{def:lambda}
\end{eqnarray}
where
\begin{eqnarray}
 |\lambda(t)|&\leq& \int_I \bigg|\int_0^{\frac x{\sqrt t}}\frac{\partial f}{\partial x}(0,z)\,dz\bigg|e^{-x^2}\,dx\leq\frac{C}{\sqrt t}\int_I |x|e^{-x^2}\,dx\nonumber\\
&\leq&\frac C{\sqrt t}.\label{estlambda}
\end{eqnarray}

\subsection{Estimation of $e^{(1-\varepsilon) t-\int_0^t\mathcal I(s)\,ds}$}

Thanks to \eqref{eq:varconst2} (and the definition of $\lambda$ and $w$: see \eqref{def:lambda} and \eqref{def:w} respectively), we see that
\begin{equation}\label{eq:alpha}
\mathcal I(t)=\frac{f(0,0)\, \int_I e^{-x^2}\,dx  +\lambda(t)}{\sqrt t}e^{(1-\varepsilon)t-\int_0^t\mathcal I(s)\,ds}+w(t),
\end{equation}
so that
\begin{eqnarray}
 e^{\int_0^t\mathcal I(s)\,ds}&=&e^{\int_0^1\mathcal I(s)\,ds}+\int_1^t\frac{d}{ds}\left(e^{\int_0^s\mathcal I(\sigma)\,d\sigma}\right)(s)\,ds\nonumber\\
&=&e^{\int_0^1\mathcal I(s)\,ds}+\int_1^t\frac{f(0,0)\, \int_I e^{-x^2}\,dx }{\sqrt s}e^{(1-\varepsilon)s}\,ds\nonumber\\
&&+\int_1^t\frac{\lambda(s)}{\sqrt s}e^{(1-\varepsilon)s}\,ds+\int_1^tw(s)e^{\int_0^s\mathcal I(\sigma)\,d\sigma}\,ds .
\label{eq:intalpha}
\end{eqnarray}
We will now estimate each of the terms on the right hand side of \eqref{eq:intalpha}. We start by estimating the third term on the right hand side, thanks to \gr{\eqref{estlambda}} and an integration by parts:
\begin{eqnarray}
\left|\int_1^t\frac{\lambda(s)}{\sqrt s}e^{(1-\varepsilon)s}\,ds\right|&\leq& C\int_1^t\frac{e^{(1-\varepsilon)s}}s\,ds\nonumber\\
&\leq& C\left[\frac{e^{(1-\varepsilon)t}}{(1-\varepsilon)t}+\int_1^t\frac{e^{(1-\varepsilon)s}}{(1-\varepsilon)s^2}\,ds\right]\nonumber\\
&\leq& \frac{C}{1-\varepsilon}\left[\frac{e^{(1-\varepsilon)t}}{t}+t\max_{s\in [1,t]}\frac{e^{(1-\varepsilon)s}}{s^2}\right]\nonumber\\
&\leq& \frac{2C}{(1-\varepsilon)t}e^{(1-\varepsilon)t},\label{eq:esttruc}
\end{eqnarray}
provided $t>0$ is large enough, and $\varepsilon>0$ is small enough (to ensure that $\max_{s\in [1,t]}\frac{e^{(1-\varepsilon)s}}{s^2}=\frac{e^{(1-\varepsilon)t}}{t^2}$).
\par
We now estimate the second term on the right hand side of \eqref{eq:intalpha}, using an integration by parts:
\begin{align*}
&\int_1^t\frac{f(0,0)\,\int_I e^{-x^2}\,dx }{\sqrt{s}}e^{(1-\varepsilon)s}\,ds \\
&\quad=f(0,0)\,\gr{\left(\int_I e^{-x^2}\,dx\right)}\,\left(\frac{e^{(1-\varepsilon)t}}{(1-\varepsilon)\sqrt{t}}-\frac {e^{1-\varepsilon}}{1-\varepsilon}+\int_1^t\frac{e^{(1-\varepsilon)s}}{2(1-\varepsilon)s^{3/2}}\,ds\right),
\end{align*}
and then, applying an estimate similar to the one used to obtain \eqref{eq:esttruc}, we get, provided that $t>0$ is large enough, and that $\varepsilon>0$ is small enough,
\begin{equation}\label{eq:est1st-term}
0\leq \int_1^t\frac{e^{(1-\varepsilon)s}}{2\,(1-\varepsilon)s^{3/2}}\,ds\leq \int_1^t\frac{e^{(1-\varepsilon)\,s}}{2\,(1-\varepsilon)\,s}\,ds
\leq\frac{1}{(1-\varepsilon)^2\,t}e^{(1-\varepsilon)\,t}.
\end{equation}
Finally, we estimate the last term of the right hand side of \eqref{eq:intalpha}, thanks to estimates \eqref{eq:est4} and \eqref{alphabound}:
\begin{eqnarray}
0\leq \int_1^tw(s)e^{\int_0^s\mathcal I(\sigma)\,d\sigma}\,ds&\leq& \int_1^t \gr{|w(s)|} e^{\|\mathcal I\|_{L^\infty(\mathbb R_+)} s}\,ds\nonumber\\
&\leq&C\,\varepsilon\int_1^t s\, e^{C\varepsilon s}\,e^{s
}\,ds\nonumber\\
&\leq& C\,\varepsilon\frac{t\,e^{\left(1+C\,\varepsilon\right)\,t}}{1+C\,\varepsilon},\label{eq:est3rd-term}
\end{eqnarray}
where we have used an integration by part to obtain the last inequality.
\par
Combining these estimates, estimate \eqref{eq:intalpha} becomes:
\begin{equation}\label{eq:expalpha1}
e^{\int_0^t\mathcal I(s)\,ds-(1-\varepsilon)t}=\frac{f(0,0)\,\int_I e^{-x^2}\,dx}{(1-\varepsilon)\sqrt t}+\mu(t),
\end{equation}
or
\begin{equation}\label{eq:expalpha2}
e^{(1-\varepsilon)t-\int_0^t\mathcal I(s)\,ds}=\left(\frac{f(0,0)\, \int_I e^{-x^2}\,dx}{(1-\varepsilon)\sqrt t}+\mu(t)\right)^{-1},
\end{equation}
where, thanks to \eqref{eq:intalpha}, \eqref{eq:esttruc}, \eqref{eq:est1st-term} and \eqref{eq:est3rd-term}, for $t\geq 1$,
\begin{equation}\label{eq:estmu}
-\frac C t\leq\mu(t)\leq C\left(\frac 1t+\varepsilon t e^{C\varepsilon t}\right).
\end{equation}

\medskip

\subsection{\gr{Estimation of} $\left|e^{(1-\varepsilon)t-\int_0^t\mathcal I(s)\,ds}-\frac{\sqrt t}{f(0,0)\, \int_I e^{-x^2}\,dx}\right|$}

Thanks to \eqref{eq:expalpha2},
\begin{align}
&\left|e^{(1-\varepsilon)t-\int_0^t\mathcal I(s)\,ds}-\frac{\sqrt t}{f(0,0)\,\int_I e^{-x^2}\,dx}\right|\nonumber\\
&\quad =\left|\left(\frac{f(0,0)\, \int_I e^{-x^2}\,dx}{(1-\varepsilon)\sqrt t}+\mu(t)\right)^{-1}-\frac{\sqrt t}{f(0,0)\, \int_I e^{-x^2}\,dx}\right|\\
&\quad=\frac{\sqrt t}{f(0,0)\, \int_I e^{-x^2}\,dx}\left|\left(\frac 1{1-\varepsilon}+\frac{\mu(t)\,\sqrt t}{f(0,0)\, \int_I e^{-x^2}\,dx}\right)^{-1}-1\right|.\label{eq:expalpha3}
\end{align}

We notice that thanks to estimate \eqref{eq:estmu},
\begin{equation}\label{eq:est1}
f(0,0)\,\int_I e^{-x^2}\,dx+(1-\varepsilon)\mu(t) \sqrt t\geq \frac {f(0,0)\,\int_I e^{-x^2}\,dx}2,
\end{equation}
as soon as $t\geq T$, for some large time $T>0$. Also, for $t\geq 1$,
\begin{equation}\label{eq:est2}
|\mu(t)|\leq C\left(\frac 1t+\varepsilon \,t\,e^{C\,\varepsilon\, t}\right).
\end{equation}
Using the bounds \eqref{eq:est1} and \eqref{eq:est2}, we can show that as soon as $t\geq T$,
\begin{eqnarray}
\left|\left(\frac 1{1-\varepsilon}+\frac{\mu(t)\sqrt t}{f(0,0)\,\int_I e^{-x^2}\,dx}\right)^{-1}-1\right|
&=&\left|\frac{-\varepsilon f(0,0)\,\int_I e^{-x^2}\,dx-(1-\varepsilon)\mu(t) \sqrt t}{f(0,0)\, \int_I e^{-x^2}\,dx+(1-\varepsilon)\mu(t) \sqrt t}\right|\nonumber\\
&\leq&C\left(\frac 1{\sqrt t}+\varepsilon t^{\frac 32}\,e^{C\,\varepsilon \, t}\right),\label{eq:est3}
\end{eqnarray}
so that identity  \eqref{eq:expalpha3} leads to the bound
\begin{equation}
\left|e^{(1-\varepsilon)t-\int_0^t\mathcal I(s)\,ds}-\frac{\sqrt t}{f(0,0)\,\int_I e^{-x^2}\,dx}\right|\leq C\left(1+\varepsilon \, t^{2}\, e^{C\,\varepsilon\, t}\right). \label{eq:est6}
\end{equation}
Notice also, as this is going to be useful further on, that for $s\geq 1$, thanks to \eqref{eq:expalpha1} and \eqref{eq:est2},
\begin{eqnarray}
\left|e^{\int_0^s\mathcal I(\sigma)\,d\sigma-(1-\varepsilon)s}-\frac{f(0,0)\, \int_I e^{-x^2}\,dx}{\sqrt s}\right|&=&\left|\mu(s)+\varepsilon\frac{f(0,0)\,\int_I e^{-x^2}\,dx}{(1-\varepsilon)\sqrt s}\right|\nonumber\\
&\leq& C\left(\frac 1{s}+\varepsilon \,s\,e^{C\,\varepsilon\, s}\right).\label{eq:est5}
\end{eqnarray}

\medskip

\subsection{Conclusion of the proof of Theorem~\ref{thm:smallt}}

In this last part of the proof, we consider times $t\geq T$. We estimate
\begin{align}
&\left\|f(t,x)-\frac{f(0,x)\,\sqrt t \,e^{-x^2 t}}{f(0,0)\,\int_I e^{-x^2}\,dx}\right\|_{L^1(I)}\nonumber\\
&\quad\leq \left\|f(t,x)-\frac{f(0,x)\,e^{-x^2 t}}{\frac{f(0,0)\,\int_I e^{-x^2}\,dx}{(1-\varepsilon)\, \sqrt t}+\mu(t)}\right\|_{L^1(I)}\nonumber\\
&\qquad +\, \Bigg\|\frac{f(0,x)\,e^{-x^2 t}}{\frac{f(0,0)\,\int_I e^{-x^2}\,dx}{(1-\varepsilon)\,\sqrt t}+\mu(t)}-\frac{f(0,x)\,\sqrt t\, e^{-x^2 t}}{f(0,0)\,\sqrt \pi}\Bigg\|_{L^1(I)}.\label{eq:diff}
\end{align}

Let us start by estimating the second term on the right hand side of \eqref{eq:diff}, thanks to
estimate \eqref{eq:est3}:
\begin{align}
&\Bigg\|\frac{f(0,x)e^{-x^2 t}}{\frac{f(0,0)\,\int_I e^{-x^2}\,dx}{(1-\varepsilon)\sqrt t}+\mu(t)}-\frac{f(0,x)\sqrt t e^{-x^2 t}}{f(0,0)\,\int_I e^{-x^2}\,dx}\Bigg\|_{L^1(I)}\nonumber\\
&\quad\leq \left\|\frac{f(0,x)\sqrt t e^{-x^2 t}}{f(0,0)\,\int_I e^{-x^2}\,dx}\left|\left(\frac 1 {1-\varepsilon}+\frac{\mu(t)\sqrt t}{f(0,0)\,\int_I e^{-x^2}\,dx}\right)^{-1}-1\right|\;\right\|_{L^1(I)}\nonumber\\
&\quad\leq \frac{f(0,x)\sqrt t }{f(0,0)\,\int_I e^{-x^2}\,dx}\left|\left(\frac 1 {1-\varepsilon}+\frac{\mu(t)\sqrt t}{f(0,0)\,\int_I e^{-x^2}\,dx}\right)^{-1}-1\right|\int_I e^{-x^2 t}\,dx\nonumber\\
&\quad \leq C \left(\frac 1{\sqrt t}+\varepsilon t^{\frac 32}\,e^{C\,\varepsilon \,t}\right)\label{eq:estprofil1}
\end{align}
We now rewrite the first term on the right hand side of \eqref{eq:diff}, using formula \eqref{eq:varconst} and \eqref{eq:expalpha2}:
\begin{align*}
&\left\|f(t,x)-\frac{f(0,x)e^{-x^2 t}}{\frac{f(0,0)\, \int_I e^{-x^2}\,dx}{(1-\varepsilon)\sqrt t}+\mu(t)}\right\|_{L^1(I)}\\
&\quad= \left\|\varepsilon \int_0^t\left(\int_I \gamma(x)f(s,y)\,dy\right)e^{(1-\varepsilon-x^2)(t-s)-\int_s^t\mathcal I(\sigma)\,d\sigma}\,ds\right\|_{L^1}\\
&\quad\leq C\varepsilon \int_I\int_0^t\mathcal I(s)e^{-x^2(t-s)}e^{(1-\varepsilon)t-\int_0^t\mathcal I(\sigma)\,d\sigma} e^{\int_0^s\mathcal I(\sigma)\,d\sigma-(1-\varepsilon)s}\,ds\,dx
\end{align*}
\gr{and then, thanks to \eqref{alphabound}, \eqref{eq:est6} and \eqref{eq:est5},
\begin{align}
&\left\|x \mapsto f(t,x)-\frac{f(0,x)e^{-x^2 t}}{\frac{f(0,0)\,\int_I e^{-x^2}\,dx}{(1-\varepsilon)\sqrt t}+\mu(t)}\right\|_{L^1(I)}\nonumber\\
&\quad \leq C\,\varepsilon \int_0^1\left(\int_I e^{-x^2(t-s)}\,dx\right) \left(\frac{\sqrt t}{f(0,0)\,\int_I e^{-x^2}\,dx}+1+\varepsilon\, t^{2}\, e^{C\,\varepsilon\, t}\right)\,ds\nonumber\\
&\qquad +C\varepsilon \int_1^t\left(\int_I e^{-x^2(t-s)}\,dx\right) \left(\frac{f(0,0)\, \int_I e^{-x^2}\,dx}{\sqrt s}+\frac 1{ s}+\varepsilon\, s \,
e^{C\,\varepsilon\, s}\right)\nonumber\\
&\phantom{dqsfesrgqdreg}\,\left(\frac{\sqrt t}{f(0,0)\,\int_I e^{-x^2}\,dx}+1+\varepsilon \,t^2\, e^{C\,\varepsilon \,t}\right)\,ds\nonumber\\
&\quad\leq C\,\varepsilon \,\frac 1{\sqrt t}\left(\sqrt t+1+\varepsilon\, t^{2}\, e^{C\,\varepsilon \,t}\right)\nonumber\\
&\qquad+C\varepsilon \int_1^t\frac 1{\sqrt {t-s}} \left(\frac 1{ \sqrt s}+\varepsilon\, s\, e^{C\,\varepsilon\, s}\right)\,\left(\sqrt t+1+\varepsilon \,t^2\, e^{C\,\varepsilon\, t}\right)\,ds.\nonumber
\end{align}
We estimate
\[\int_1^t\frac{s e^{C\varepsilon s}}{\sqrt{t-s}}\,ds\leq t e^{C\varepsilon t}\int_1^t\frac{ds}{\sqrt{t-s}}\leq Ct^{\frac 32} e^{C\varepsilon t},\]
and then
\begin{align}
&\left\| x \mapsto f(t,x)-\frac{f(0,x)e^{-x^2 t}}{\frac{f(0,0)\,\int_I e^{-x^2}\,dx}{(1-\varepsilon)\sqrt t}+\mu(t)}\right\|_{L^1(I)}\nonumber\\
&\quad\leq C\varepsilon \left(1+\frac 1{\sqrt t}+\varepsilon\, t^{\frac 3 2}\, e^{C\,\varepsilon\, t}\right)\nonumber\\
&\qquad+C\,\varepsilon \,\left(1+\varepsilon t^{\frac 32} e^{C\varepsilon t}\right)\, \left(\sqrt t+1+\varepsilon\, t^2 \,e^{C\,\varepsilon\, t}\right)\nonumber\\
&\quad\leq C\left(\varepsilon+\varepsilon\sqrt t+\frac \varepsilon{\sqrt t}+ \left(\varepsilon t^{\frac 32}+\varepsilon^2 t^2+\varepsilon^2t^{\frac 32}+\varepsilon^3t^{\frac 72}\right)e^{C\varepsilon t}\right)\nonumber\\
&\quad\leq C\left(\frac \varepsilon{\sqrt t}+ \varepsilon t^{\frac 32}e^{C\varepsilon t}\right),\label{eq:estprofil2}
\end{align}
where we have used the fact that $\varepsilon t\leq Ce^{C\varepsilon t}$. Thanks to \eqref{eq:estprofil1} and \eqref{eq:estprofil2}, \eqref{eq:diff} becomes:
\begin{equation*}
\Bigg\|x \mapsto f(t,x)-\frac{f(0,x)\sqrt t e^{-x^2 t}}{f(0,0)\,\int_I e^{-y^2}\,dy}\Bigg\|_{L^1(I)}\leq C \left(\frac 1{\sqrt t}+\varepsilon\, t^{\frac 3 2}\,e^{C\,\varepsilon\, t}\right).
\end{equation*}
Theorem~\ref{thm:smallt} follows from this estimate.}

\subsection{Proof of Corollary~\ref{cor:smallt}}

If we assume that $t\in\left[\frac 1{\kappa^2}, \kappa^{\frac 23} \varepsilon^{-\frac 23}\right]$, then \eqref{est:final-smallt} becomes

\begin{eqnarray*}
\Bigg\| x \mapsto f(t,x)-\frac{f(0,x)\sqrt t e^{-x^2 t}}{f(0,0)\, \int_I e^{-y^2}\,dy}\Bigg\|_{L^1(I)}&\leq&  C \left(\kappa+\kappa\,e^{C\, \kappa^{\frac 23} \varepsilon^{\frac 13}}
\right),
\end{eqnarray*}
%
%
and if furthermore $\varepsilon\leq \kappa\leq 1$, then
\begin{eqnarray*}
\Bigg\|x \mapsto f(t,x)-\frac{f(0,x)\,\sqrt t\, e^{-x^2 t}}{f(0,0)\, \int_I e^{-y^2}\,dy}\Bigg\|_{L^1(I)}&\leq& C\,\kappa,
\end{eqnarray*}
which proves Corollary~\ref{cor:smallt}\gr{, provided that $\kappa>0$ is small enough.}

\medskip

{\bf{Acknowledgement}}:  The research leading to this paper was funded by 
the French ``ANR blanche'' project Kibord: ANR-13-BS01-0004, and  ``ANR JCJC'' project MODEVOL ANR-13-JS01-0009. A.C. and S.C. were partially supported by the Ministerio de Ciencia e Innovación,
Grants MTM2011-27739-C04-02 and MTM2014-52402-C3-2-P.


\begin{thebibliography}{99}

\bibitem{Alfaro}M. Alfaro, R. Carles, Explicit solutions for replicator-mutator equations: extinction vs. acceleration. \emph{SIAM J. Appl. Math.} {\bf 74}, 1919--1934 (2014).

\bibitem{nagel} W. Arendt, A, Grabosch, G. Greiner, I. Groh, H.P. Lotz,  U. Moustakas, R. Nagel, F. Neubrander, U. Schlotterbeck. One-parameter semigroups of positive operators. Lecture Notes in Mathematics, 1184. Springer-Verlag, Berlin, 1986.

\bibitem{Bell} G. Bell, A. Gonzalez, Evolutionary rescue can
prevent extinction following environmental change.
\emph{Ecol. Lett.} {\bf 12}, 942--948 (2009).

\bibitem{Bolnick} D. I. Bolnicke, P. Amarasekare, M. S. Ara\'ujo, R. B\"urger, J. M. Levine, M. Novak, V. H. W. Rudolf, S. J. Schreiber, M. C. Urban, D. A. Vasseur, Why intraspecific trait variation matters in community ecology. \emph{Trends Ecol. Evol.} 26, 183--192 (2011).

\bibitem{Bulmer} M. G. Bulmer, The mathematical theory of quantitative genetics. Oxford, UK: Clarendon Press (1980).

\bibitem{Burger2} R. B\" urger, I. M. Bomze, Stationary distributions under mutation-selection balance: Structure and properties. \emph{Adv. in Appl. Probab.}, {\bf 28}, 227--251 (1996).

\bibitem{Burger} R. B\" urger, The mathematical theory of selection, recombination, and mutation. Wiley Series in Mathematical and Computational Biology. John Wiley \& Sons, Ltd., Chichester, 2000.

\bibitem{Calsina}  A. Calsina, S. Cuadrado, L. Desvillettes, G. Raoul, Asymptotics of steady states of a selection mutation equation for small mutation rate. \emph{P. Roy. Soc. Edinb. A } {\bf 143}(06), 1123--1146 (2013).

\bibitem{Canton} R. Cant\' on, M. I. Morosini, Emergence and spread of antibiotic resistance following exposure to antibiotics. \emph{FEMS microbiol. Rev.} {\bf 35}(5), 977--991 (2011).

\bibitem{Carlson}S. M. Carlson, C. J. Cunningham, P. A. Westley, Evolutionary rescue in a changing world. \emph{Trends Ecol. Evol.} {\bf 29}(9), 521--530 (2014).

\bibitem{Chevin}L. M. Chevin, R. Gallet, R., Gomulkiewicz, R. D. Holt, S. Fellous, Phenotypic plasticity in evolutionary rescue experiments. Phil. Trans. R. Soc. B: Biological Sciences, {\bf 368}(1610), 20120089 (2013).

\bibitem{Coville} J. Coville, Convergence to the equilibrium of positive solution of some  mutation-selection model. preprint arXiv:1308.6471.

\bibitem{Champagnat}N. Champagnat, R. Ferri\`ere, S. M\'el\'eard, Unifying evolutionary dynamics: from individual stochastic processes to macroscopic models. \emph{Theor. Popul. Biol.} {\bf 69}(3), 297--321 (2006).

\bibitem{clement} Ph. Cl\'{e}ment, H. J. A. M. Heijmans, S. Angenent, C. J. van Duijn, B. de Pagter, One-parameter semigroups. CWI Monographs, 5. North-Holland Publishing Co., Amsterdam, 1987.

\bibitem{Diekmann}O. Diekmann, J. A. P. Heesterbeek, J. A. Metz, On the definition and the computation of the basic reproduction ratio $R_0$ in models for infectious diseases in heterogeneous populations. J. Math. Biol. {\bf 28}(4), 365-382 (1990).

\bibitem{Fisher}R. A. Fisher, The Genetical Theory of Natural Selection, Clarendon Press, Oxford 1930.

\bibitem{Gandon}S. Gandon, M. E. Hochberg, R. D. Holt, T. Day, What limits the evolutionary emergence of pathogens? Phil. Trans. R. Soc. B {\bf 368}(1610), 20120086 (2013).

\bibitem{Gonzales}A. Gonzalez, O. Ronce, R. Ferriere, M. E. Hochberg, Evolutionary rescue: an emerging focus at the intersection between ecology and evolution. \emph{Phil. Trans. R. Soc. B}, {\bf 368}(1610), 20120404 (2013).

\bibitem{Hastings}A. Hastings, Transients: the key to long-term ecological understanding? \emph{Trends Ecol. Evol.} {\bf 19}(1), 39--45 (2004).

\bibitem{Hughes} A. R. Hughes, B. D. Inouye, M. T. J. Johnson, N. Underwood, M. Vellend, Ecological consequences of genetic diversity. \emph{Ecol. Lett.} {\bf 11}, 609--623 (2008).

\bibitem{Kimura} M. Kimura, A stochastic model concerning the maintenance of genetic variability in quantitative characters. \emph{Proc. Natl. Acad. Sci. U.S.A} {\bf 54}, 731--736 (1965).

\bibitem{Kingsolver} J. Kingsolver, D. Pfennig,  Patterns and power of phenotypic selection in nature. \emph{Bioscience} {\bf 57}, 561--572 (2007).

\bibitem{Lush} J. L. Lush, Animal Breeding Plans. Ames, Iowa: Iowa State Press, 1937.

\bibitem{Luria} S. E. Luria, M. Delbr{\"u}ck, Mutations of bacteria from virus sensitivity to virus resistance. \emph{Genetics} {\bf 28}(6), 491--511 (1943).

\bibitem{Mather}K. Mather, J. L. Jinks. Biometrical genetics (2 ed.). London: Chapman \& Hall (1971).

\bibitem{Martin}G. Martin, R. Aguil\'ee, J. Ramsayer, O. Kaltz, O. Ronce, The probability of evolutionary rescue: towards a
quantitative comparison between theory and evolution
experiments. \emph{Phil. Trans. R. Soc. B} {\bf 368}, 20120088 (2012).

\bibitem{Toprak}E. Toprak, A. Veres, J. B. Michel, R. Chait, D. L. Hartl, R. Kishony, Evolutionary paths to antibiotic resistance under dynamically sustained drug selection. \emph{Nat. Genet.}, {\bf 44}(1), 101--105 (2012).

\bibitem{Turelli} M. Turelli, N. H. Barton, Genetic and statistical analyses of strong selection on polygenic traits: what, me normal? \emph{Genetics} {\bf 138}, 913--941 (1994).

\bibitem{Vellend}M. Vellend, The consequences of genetic diversity in
competitive communities. \emph{Ecology} {\bf 87}, 304--311 (2006).

\bibitem{Violle}C. Violle, B. J. Enquist, B. J. McGill, L. Jiang, C. H. Albert, C. Hulshof, V. Jung, J. Messier, The return of the variance: intraspecific variability in community ecology. \emph{Trends Ecol. Evol.} {\bf 27}(4), 244--52 (2012).

\end{thebibliography}
\end{document}